\documentclass{siamart190516}

\usepackage{amssymb}

\newsiamremark{remark}{Remark}
\newsiamthm{example}{Example}
\numberwithin{theorem}{section}

\usepackage{etex}
\usepackage{subfig}
\usepackage{amsmath}
\usepackage{amssymb,color,enumerate}
\usepackage[latin1]{inputenc}
\usepackage{xspace}
\usepackage{algorithm}
\usepackage{comment}
\usepackage{algorithmic}
\usepackage{tikz}
\usepackage{graphicx}
\usepackage[backend=biber, style=ieee]{biblatex}
\usepackage{multirow}
\usepackage{booktabs}

\usepackage[normalem]{ulem} 

\usepackage{pgfplots}
\pgfplotsset{compat=1.13}

\def\norm#1{\left\|#1\right\|}
\newcommand{\VECTOR}[3]{\left(#1,\, #2,\, #3\right)\transp}
\newcommand{\R}{\mathbb R}
\newcommand{\N}{\mathbb N}

\newcommand{\ybar}{y_d}

\newcommand{\efeas}{\varepsilon_{feas}}

\newcommand{\Jt}{J}

\newcommand{\pmax}{p_{\max}}
\newcommand{\transp}{^{^{\scriptstyle\intercal}}} 
\newcommand{\Ot}{\tilde{\Omega}}
\newcommand{\Chdist}{\mathrm{dist}_\infty}
\DeclareMathOperator*{\argmin}{arg\,min}
\newcommand{\BnB}{branch-and-bound\xspace}
\newcommand{\cplexmiqp}{{\tt cplexmiqp}\xspace}
\newcommand{\dgopt}{$\delta$-global optimizer\xspace}

\newcommand{\REV}[1]{\textcolor{black}{#1}}

\def\minres{{\large {\sc minres}}\xspace}
\def\gmres{{\large {\sc gmres}}\xspace}

\title{Improved penalty algorithm for Mixed Integer PDE Constrained Optimization Problems}

\author{
  Dominik Garmatter\thanks{Department of Mathematics, Chemnitz University of Technology, Reichenhainer Str.41, Chemnitz, 01926, Germany (\email{dominik.garmatter@math.tu-chemnitz.de}, \email{martin.stoll@math.tu-chemnitz.de})}
  \and
  Margherita Porcelli\thanks{Department of Mathematics, University of Bologna, Piazza di Porta San Donato 5, Bologna, 40126, Italy (\email{margherita.porcelli@unibo.it})}
  \and
  Francesco Rinaldi\thanks{Department of Mathematics "Tullio Levi-Civita", University of Padova, via Trieste 63, Padova, 35121, Italy (\email{rinaldi@math.unipd.it})}
  \and
  Martin Stoll\footnotemark[1]
}

\headers{Improved penalty algorithm for MIPDECO Problems}{D. Garmatter, M. Porcelli, F. Rinaldi and M. Stoll}

\begin{document}

\maketitle

\begin{abstract}
Optimal control problems including partial differential equation (PDE) as well as integer constraints merge the combinatorial difficulties of integer programming and the challenges related to large-scale systems resulting from discretized PDEs. So far, the \BnB framework has been the most common solution strategy for such problems. In order to provide an alternative solution approach, especially in a large-scale context, this article investigates penalization techniques. Taking inspiration from a well-known family of existing exact penalty algorithms, a novel \textit{improved penalty algorithm} is derived, whose key ingredients are a basin hopping strategy and an interior point method, both of which are specialized for the problem class.
A thorough numerical investigation is carried out for a standard stationary test problem. Extensions to a convection-diffusion as well as a nonlinear test problem finally demonstrate the versatility of the approach.

\end{abstract}

\begin{keyword}
mixed integer optimization, optimal control, PDE-constrained optimization, exact penalty methods, interior point methods

\end{keyword}


\section{Introduction}
\label{sec:1_MINLPpaper}

Optimal control problems that are governed by a partial differential equation (PDE) as well as integer constraints on the control and possible additional constraints are commonly referred to as mixed integer PDE-constrained optimization (MIPDECO) problems. They pose several challenges as they combine two fields that have been surprisingly distinct from each other in the past: integer programming and PDEs. While integer optimization problems have an inherent combinatorial complexity that has to be dealt with, PDE-constrained optimization problems have to deal with possibly large-scale linear systems resulting from the discretization of the PDE, see, e.g., \cite{troltzsch2010optimal}.

In spite of these challenges, MIPDECO problems are gaining increased attention as they naturally arise in many real world applications such as gas networks \cite{hahn2017mixed,Schewe_Martin_2015}, the placement of tidal and wind turbines \cite{FUNKE2014658,Zhang2014,wesselhoeft2017mixed} or power networks \cite{Goettlich2019}. From the theoretical point of view, there have been  recent advances in the field including a Sum-up-Rounding strategy \cite{manns2018multi, LeyfferSUR}, a derivative-free approach \cite{larson2019method}, and new sophisticated rounding techniques \cite{LeyfferINV}.

A classical solution approach for a MIPDECO problem is to \textit{first-discretize-then-optimize}, where the PDE and the control are discretized such that the continuous MIPDECO problem is then approximated by a finite-dimensional (and possibly large-scale) mixed-integer nonlinear programming problem (MINLP). Standard techniques, see, e.g., \cite{BnB_Masterpaper} for an excellent overview, such as \BnB can then be used to solve the MINLP. Unfortunately, depending on the size of the finite\REV{-}dimensional approximation, these techniques may struggle. On the one hand, the discretization of the control might result in a large amount of integer variables and thus an immense combinatorial complexity of the MINLP. On the other hand, the discretization of the PDE results in large-scale linear systems occurring whenever an NLP-relaxation of the MINLP has to be solved.

The contribution of this article to the field is to provide an alternative approach for MIPDECO problems via an equivalent penalty formulation of the original problem. While penalty reformulations have been studied in the context of integer programming, see, e.g., \cite{giannessi1976connections,Lucidi_2010,Rinaldi_2009,zhu2003penalty}, and penalty approaches have been developed, see, e.g., \cite{Costa_2016,Lucidi_2011,Murray_2008}, there have been (to the knowledge of the authors) no contributions that explicitly deal with MIPDECO problems.

The general idea of penalty reformulations is to relax the integer constraints of the problem and add a suitable penalty term to the objective function, thus penalizing controls that violate the previously present integer constraints. A naive solution strategy \REV{is} to iteratively solve the resulting penalty formulation while increasing the amount of penalization in each iteration until one ends up with an integer solution. The upside of such penalization strategies is that the combinatorial complexity of the integer constraints is eliminated from the problem formulation and the penalty term then ensures that the resulting solution satisfies the integer constraints. The downside is that penalty terms are usually concave such that one has to deal with non-convex NLPs with  a possibly exponential amount of local minimizers.

To still provide qualitative solutions in this context, the main contribution of this article is the development of a novel algorithm that is closely related to a family of existing exact penalty (EXP) algorithms, which have been analyzed both in the context of general constrained optimization \cite{di2012approach,di2015derivative} and in the context of integer optimization \cite{Lucidi_2011}. Roughly speaking, a general EXP algorithmic framework, which is an iterative procedure, provides an automatic tool for when to increase penalization and when to aim for a better minimizer via a suitable global solver for the penalized subproblems. One can then show convergence towards a global minimizer of the original problem, see, e.g., \cite[Corollary 1]{Lucidi_2011} for the analysis of the integer case. 

A practical implementation of an EXP algorithm is carried out in this paper. Although the algorithm is developed taking into account a model problem, it will become clear that it can handle quite general MIPDECO problems. The idea of the resulting  improved penalty algorithm (IPA) is to combine the EXP framework with a suitably developed search approach, closely connected to basin hopping or iterated local search methods, see, e.g., \cite{grosso2007population,leary2000global}. The search combines a local optimization algorithm with a perturbation strategy (both tailored to the specific application) in order to find either the global or a good local minimum of the penalty reformulation. 

Our suitably developed local optimization solver is an interior point method that exploits the structure of  the penalty formulation related to a MIPDECO problem in the following ways:
\begin{itemize}
	\item it explicitly handles the non-convexity introduced by the penalty term;
	\item it uses a specific preconditioner to efficiently handle the linear algebra. 
\end{itemize} 
Via this approach, large-scale problems can be handled and the IPA is numerically compared, both for a standard test problem and a convection-diffusion problem, to a traditional penalty method as well as a \BnB routine from CPLEX \cite{CPLEX}.

The remainder of this work is organized as follows: the model problem is presented and discretized in Section \ref{sec:2_MINLPpaper}. Section \ref{sec:3_MINLPpaper} reviews the EXP algorithm, extends its convergence theory to the class of MIPDECO problems considered, and then develops the novel improved penalty algorithm. Section \ref{sec:4_MINLPpaper} gathers implementation details of the IPA, such as the interior point method, and briefly collects the remaining algorithms for the numerical comparison that is carried out in Section \ref{sec:5_MINLPpaper}. Finally, conclusions are drawn in Section \ref{sec:6_MINLPpaper} including an outlook on MIPDECO problems with a nonlinear PDE constraint.

\section{Problem formulation}
\label{sec:2_MINLPpaper}

We begin with the description of the optimal control model problem in function spaces. Following the first-discretize-then-optimize approach, we then present the discretized model problem as well as its continuous relaxation. Finally, we review existing solution techniques.

\subsection{\REV{Binary optimal control problem}}
\label{sec:2_1_MINLPpaper}

We begin with the description of the PDE in order to formulate the optimal control problem. Consider a bounded domain $\Omega\subset\R^2$ with Lipschitz boundary, source functions $\phi_1,\dots ,\phi_l\in L^2(\Omega)$ and based on these the PDE: for a given control vector $u=(u_1,\dots , u_l)\transp\in\R^l$ find the state $y\in H_0^1(\Omega)$ solving
\begin{align}
\label{eq:PDE_MINLPpaper}
-\Delta y(x) = \sum_{i=1}^l u_i \phi_i(x),\quad x\in\Omega,
\end{align}
where the PDE is to be understood in the weak sense.
Existence and uniqueness of a solution $y\in H_0^1(\Omega)$ of \eqref{eq:PDE_MINLPpaper} follow from the Lax\REV{--}Milgram theorem.
For now, we choose to model the sources $\phi_1,\dots ,\phi_l$ as Gaussian functions with centers 
\REV{$\tilde{x}_1,\dots,\tilde{x}_l$ in the interior of $\Omega$.}
Thus, for $x\in\R^2$,
\begin{align}
\label{eq:gaussian_sources_MINLPpaper}
\phi_i(x) := \kappa e^{-\frac{\norm{x-\tilde{x}_i}_2^2}{\omega}},\quad i=1,\dots, l,
\end{align}
with height $\kappa >0$ and width $\omega>0$. 
The optimal control problem in function spaces then reads: given a desired state $\ybar\in L^2(\Omega)$, find a solution pair $(y,u)\in H^1_0(\Omega)\times \{0,1\}^l$ of 
\begin{align}
\begin{array}{cl}
\label{eq:MINLP_cont_MINLPpaper}
\displaystyle
\min_{y\in H^1_0(\Omega),u\in\{0,1\}^l} & \frac{1}{2}\norm{y-\ybar}_{L^2(\Omega)}^2,\\
\mbox{s.t.} & (y,u) \mbox{ fulfill } \eqref{eq:PDE_MINLPpaper},\quad\mbox{and}\quad\sum_{i=1}^l u_i \leq S\in\N,
\end{array}
\end{align}
where the inequality constraint in \eqref{eq:MINLP_cont_MINLPpaper} is commonly referred to as a \textit{knapsack constraint}.
This problem can be interpreted as fitting a desired heating pattern $\ybar$ by activating up to $S$ many sources that are distributed around the domain $\Omega$. Since the \REV{set of feasible controls $\{0,1\}^l$} 
is finite and for each control there is a uniquely determined state $y$, problem \eqref{eq:MINLP_cont_MINLPpaper} is \REV{essentially} a combinatorial problem so that existence of at least one global minimizer is guaranteed.
We close this section with some remarks on the presented model problem.

\begin{remark}
	\label{rem:sources_MINLPpaper}
	\begin{enumerate}[(a)]
		\item The Gaussian source functions are motivated by porous-media flow applications to determine the number of boreholes, see, e.g., \cite{Fipke_2008,Ozdogan_2006}, and problem \eqref{eq:MINLP_cont_MINLPpaper} with this choice is furthermore a model problem mentioned in \cite[Section 19.3]{leyffer2016optimization}. We will see throughout the development of our algorithm that it does not rely on this particular modelling of the control. Exemplarily, Section \ref{sec:5_2_MINLPpaper} will deal with a convection-diffusion equation with piecewise constant source \REV{functions $\phi_i$} (and we mention that piecewise constant source \REV{functions} were also used in \cite{Buchheim_2018})\REV{. 
		\item Having a fixed number of $l$ source functions results in the number of integer variables being independent of the discretization mesh. While the algorithm presented in this article could in principle handle a general distributed control (as for exmaple proposed in \cite{LeyfferINV}) such that the amount of controls would then scale with the physical space discretization, we note that the overall problem would be more challenging.}
		\item It is well-known that problems with general integer constraints can be reduced to problems with binary constraints, see, e.g., \cite{Giannessi_1998}. Furthermore, \cite[Section 4]{Lucidi_2010} provides an alternative in the context of penalty approaches by directly penalizing general integer constraints. 
	\end{enumerate}
\end{remark}

\subsection{Discretized model problem and continuous relaxation}
\label{sec:2_2_MINLPpaper}

\REV{We introduce a conforming mesh over $\Omega$ with $N$ vertices such that, after choosing a suitable finite element space, $M\in\R^{N\times N}$ and $K\in\R^{N\times N}$ denote the mass and stiffness matrices and we note that $K$ and $M$ are positive definite and $M$ is symmetric. We refer to \cite{deckelnick2012note} for a discussion on the convergence of the discretized quantities.}
Furthermore, let the matrix $\Phi\in\R^{N\times l}$ contain the finite element coefficients of the source functions in its columns, i.e., each column contains the evaluation of the respective source function at the $N$ vertices of the grid. With these matrices at hand, we formulate the \textit{discretized optimal control problem}
\begin{align}
\begin{array}{cl}
\label{eq:MINLP_raw_MINLPpaper}	
\displaystyle
\min_{y\in\R^N, u\in\{0,1\}^l} & \frac{1}{2}(y-\ybar)\transp M (y-\ybar),\\
\mbox{s.t.} & \quad Ky = M\Phi u,\quad \mbox{and}\quad\sum_{i=1}^l u_i \leq S\in\N.
\end{array}
\end{align}
In \eqref{eq:MINLP_raw_MINLPpaper} and for the remainder of this article, $y$ denotes the vector of the finite element coefficients of the corresponding finite element approximation of \eqref{eq:PDE_MINLPpaper} rather than the actual PDE-solution. The same holds true for the desired state $\ybar$\REV{,} which from now on represents a finite element coefficient vector instead of an actual $L^2(\Omega)$-function. 
Relaxing the integer constraints in \eqref{eq:MINLP_raw_MINLPpaper} yields the \textit{continuous relaxation}
\begin{equation}
\begin{array}{cl}
\label{eq:MINLP_raw_contrelax_MINLPaper}
\displaystyle
\min_{y\in\R^N,~u\in\R^l} & \frac{1}{2}(y-\ybar)\transp M (y-\ybar)\\
\mbox{s.t.} & \quad Ky = M\Phi u,\quad 0\leq u\leq 1,\quad\mbox{and}\quad\sum_{i=1}^l u_i \leq S\in\N.
\end{array}
\end{equation}
We reformulate both problems \eqref{eq:MINLP_raw_MINLPpaper} and \eqref{eq:MINLP_raw_contrelax_MINLPaper} in a more compact way.

\begin{lemma}
	\label{lemma:X_and_W_MINLPpaper}
	Introducing for $x\in\R^{N+l}$
	$$
	\tilde{J}(x) := \frac{1}{2} x\transp \begin{bmatrix} M & 0 \\ 0 & 0 \end{bmatrix} x - x\transp \begin{bmatrix} M\ybar \\ 0 \end{bmatrix} + \frac{1}{2}\ybar\transp M\ybar
	$$
	and $f:\R^l\to\R^N:u\mapsto K^{-1}M\Phi u$,
	problems \eqref{eq:MINLP_raw_MINLPpaper} and \eqref{eq:MINLP_raw_contrelax_MINLPaper} are equivalent to
	\begin{equation}
	\label{eq:MINLP_MINLPpaper}
	\tag{P}
	\min_{x\in W} \tilde{J}(x) \quad W := \bigg\{x = (y,u)\transp\in\R^{N+l} \mathrel{\Big|} u\in \{0,1\}^l,~\sum_{i=1}^ l u_i \leq S,~ y = f(u)\bigg\}
	\end{equation}
	and
	\begin{equation}
	\label{eq:MINLP_contrelax_MINLPaper}
	\tag{Pcont}
	\min_{x\in X} \tilde{J}(x) \quad X := \bigg\{x = (y,u)\transp\in\R^{N+l} \mathrel{\Big|} u\in [0,1]^l,~\sum_{i=1}^ l u_i \leq S,~ y = f(u)\bigg\},
	\end{equation}
	respectively. $W\subset\R^{N+l}$ is a compact set and $X\subset\R^{N+l}$ is compact and convex such that \eqref{eq:MINLP_contrelax_MINLPaper} is a convex problem. 
\end{lemma}
\begin{proof}
	The equivalence of the problems in question follows from the definition of the sets $W$ and $X$ and the map $f$. 
	Furthermore, $W$ is obviously compact and $X$ as the image of a compact convex set under a linear map is compact and convex.
	Thus, the convexity of \eqref{eq:MINLP_contrelax_MINLPaper} follows from the convexity of $X$ and the convexity of $\tilde{J}$ where the matrix 
	$
	\begin{bmatrix} M & 0 \\ 0 & 0 \end{bmatrix},
	$
	with $M$ being positive definite, is positive semidefinite.
\end{proof}

\noindent The authors acknowledge that \eqref{eq:MINLP_MINLPpaper} might be tackled by existing methods, see, e.g., \cite{Buchheim_2018,Costa_2016,Murray_2008}, and thus want to comment on the limitations of these approaches in a large-scale context. 
\begin{enumerate}
	\item In \cite{Buchheim_2018}, a branch-and-cut algorithm is presented, where the computation of a cutting plane requires one linear PDE solution per dimension of the control space. Therefore, this approach can become excessively time-consuming for large $l$. 
	\item In \cite{Costa_2016}, an EXP framework that embeds an iterative genetic algorithm is presented, where the amount of objective function evaluations per iteration \REV{of the genetic algorithm in \cite{Costa_2016} scales} 
	quadratically with the problem dimension $l$. But in the PDE-constrained optimization context of \eqref{eq:MINLP_MINLPpaper} an evaluation of the objective function requires a PDE solution, such that the approach can become costly for large $l$ and/or $N$.
	\item In \cite{Murray_2008}, a penalty-based approach combined with a smoothing method is considered to solve nonlinear and possibly non-convex optimization problems with binary variables. The main drawback in this case is: there is no theoretical guarantee that one converges towards the global minimum. Hence, the smoothing and penalty parameters need to be carefully initialized and handled during the optimization process in order to avoid getting stuck in bad local minima.
	\item Lastly, a comparison of our method towards strategies such as a Sum-Up Rounding method for PDEs \cite{LeyfferSUR} and a sophisticated rounding technique \cite{LeyfferINV} are topics for future work.
\end{enumerate}
Although not strictly considering mixed-integer problems, we mention that the multi-bang approach described in \cite{clason2014multi} may also be considered to solve MIPDECO problems. 
In view of \eqref{eq:MINLP_MINLPpaper}, it is not clear how the approach from \cite{clason2014multi} translates from distributed to modeled controls and how the addition of a knapsack constraint can be dealt with.

\section{Improved penalty algorithm (IPA)}
\label{sec:3_MINLPpaper}

This section contains the main contribution of this article, the development of our novel improved penalty algorithm (IPA). We will first introduce a well-known equivalent penalty reformulation of \eqref{eq:MINLP_MINLPpaper}, followed by an exact penalty algorithm from \cite{Lucidi_2011}. Afterwards we will develop the IPA, where the idea is to combine the EXP framework with a local search strategy such that the resulting algorithm only relies on a local solver.

\subsection{Penalty formulation and exact penalty (EXP) algorithm}
\label{sec:3_1_MINLPpaper}

Starting from the continuous relaxation \eqref{eq:MINLP_raw_contrelax_MINLPaper}, we add the well-known penalty term 
\begin{equation}
\label{eq:penaltyterm_MINLPpaper}
\frac{1}{\varepsilon} \sum_{i=1}^l u_i(1-u_i)
\end{equation}
to the objective function. Obviously, this concave penalty term penalizes a non-binary control, where $\varepsilon > 0$ controls the amount of penalization. This yields the following \textit{penalty formulation}
\begin{equation}
\begin{array}{cl}
\label{eq:MINLP_raw_penalty_MINLPpaper}
\displaystyle
\min_{y\in\R^N,~u\in\R^l} & \frac{1}{2}(y-\ybar)\transp M (y-\ybar) + \frac{1}{\varepsilon} \sum_{i=1}^l u_i(1-u_i)\\
\mbox{s.t.} & Ky = M\Phi u,\quad 0\leq u\leq 1\quad\mbox{and}\quad\sum_{i=1}^l u_i \leq S\in\N.
\end{array}
\end{equation}
Following Lemma \ref{lemma:X_and_W_MINLPpaper}, \eqref{eq:MINLP_raw_penalty_MINLPpaper} can be rewritten as
\begin{equation}
\label{eq:MINLP_penalty}
\tag{Ppen}
\begin{split}
&\min_{x\in X} \Jt(x;\varepsilon),\quad\mbox{with}\\
&\Jt(x;\varepsilon) := \frac{1}{2} x\transp \begin{bmatrix} M & 0 \\ 0 & -\frac{2}{\varepsilon}I_{l} \end{bmatrix} x - x\transp \begin{bmatrix} M\ybar \\ -\frac{1}{\varepsilon}{\bf 1} \end{bmatrix} + \frac{1}{2}\ybar\transp M\ybar, 
\end{split}
\end{equation}
where $I_{l}\in\R^{l\times l}$ is the identity-matrix and ${\bf 1} := (1, \dots, 1)\transp \in \R^l$.
\begin{proposition}
	\label{prop:Equiv_MINLPpaper}
	There exists an $\tilde{\varepsilon} > 0$ such that for all $\varepsilon \in (0,\tilde{\varepsilon}]$ problems \eqref{eq:MINLP_MINLPpaper} and \eqref{eq:MINLP_penalty} have the same minimizers. Having the same minimizers here means that both problems \eqref{eq:MINLP_MINLPpaper} and \eqref{eq:MINLP_penalty} have the same global minima (if there exist multiple). In this sense both problems \eqref{eq:MINLP_MINLPpaper} and \eqref{eq:MINLP_penalty} are equivalent.
\end{proposition}
\begin{proof}
	From Lemma \ref{lemma:X_and_W_MINLPpaper} it is clear that $J \in C^1(\R^{N+l})$ and that $W$ and $X$ are compact. Together with the results derived in \cite[Section 3]{Lucidi_2010} all assumptions of \cite[Theorem 2.1]{Lucidi_2010} are fulfilled such that the desired statement follows.
\end{proof}
\noindent We mention that the equivalence result from Proposition \ref{prop:Equiv_MINLPpaper} also holds for a variety of concave penalty terms, see, e.g., \cite[\REV{Equations} (19)-(23)]{Lucidi_2010} or \cite[\REV{Equation} (21)]{Rinaldi_2009}. We chose the penalty term \eqref{eq:penaltyterm_MINLPpaper} in this article since it is quadratic and thus the combined objective function $\Jt$ remains quadratic.

Before we formulate the exact penalty algorithm, we introduce a rounding strategy that suitably handles the knapsack constraint in $X$ and $W$ and prove that it is the correct tool required for the algorithm design.
\begin{definition}
	\label{def:SR_MINLPpaper}
	For $x = [y,u]\transp\in X$ and $S\in\N$, with $S\leq l$, let $u_S\in\R^S$ denote the $S$ largest components of $u$. The \textit{smart rounding} of $x$ is given as follows:
	$$[x]_{SR} := \left(f([u]_{SR}), [u]_{SR}\right)\transp \in W,$$
	with $f$ defined as in Lemma~\ref{lemma:X_and_W_MINLPpaper} and $\left[u \right]_{SR}$ obtained by rounding $u_S$ component-wise to the closest integer, while setting the remaining components to $0$.
\end{definition}

\noindent We illustrate the smart rounding by considering a simple example working only with control values\REV{: it will demonstrate that the smart rounding does, by definition, satisfy the knapsack constraint, while the usual rounding to the closest integer may fail to do so.}

\begin{example}
	Let $S=2$ and $l=3$ and let $[\cdot]$ denote the usual rounding to the closest integer. Then, for 
	\[u = \VECTOR{0.8}{0.7}{0.1}\quad\mbox{and}\quad v = \VECTOR{0.63}{0.61}{0.62}\]
	it is $u_S = \left(0.8,\, 0.7\right)\transp$ such that $[u]_{SR} = \VECTOR{1}{1}{0} \REV{= [u]}$, whereas $v_S = \left(0.63,\, 0.62\right)\transp$ such that $[v]_{SR} = \VECTOR{1}{0}{1} \neq \REV{[v] = \VECTOR{1}{1}{1}}$. 
\end{example}

\noindent With \REV{Definition \ref{def:SR_MINLPpaper}} at hand, we state in Algorithm \ref{algo:EXP_Algorithm_MINLPpaper} the adaptation of the EXP algorithm from \cite[Section \REV{3}]{Lucidi_2011}  to our model problem \eqref{eq:MINLP_penalty}. 
\REV{We note that Algorithm \ref{algo:EXP_Algorithm_MINLPpaper} is obtained from the original EXP algorithm following the ideas of \cite[Section 4]{Lucidi_2011}, where an adaptation for bound-constrained mixed integer problems was defined.}

\begin{algorithm}[H]
	\caption{EXP($\varepsilon^0 >0$, $\delta^0 >0$, $\sigma\in (0,1)$)}
	\begin{algorithmic}[1]
		\label{algo:EXP_Algorithm_MINLPpaper}
		\STATE{$n = 0$, $\varepsilon^n = \varepsilon^0$, $\delta^n = \delta^0$}
		\STATE{\textbf{Step 1.} Compute $x^n \in X$ such that $\Jt(x^n;\varepsilon^n) \leq \Jt(x;\varepsilon^n) + \delta^n$ for all $x\in X$.}
		\STATE{\textbf{Step 2.}}
		\IF{$x^n\notin W$ \textbf{and} $\Jt(x^n;\varepsilon^n) - \Jt([x^n]_{SR};\varepsilon^n) \leq \varepsilon^n \norm{x^n - [x^n]_{SR}}_2$}
		\STATE{$\varepsilon^{n+1} = \sigma\varepsilon^n$, $\delta^{n+1} = \delta^n$}
		\ELSE
		\STATE{$\varepsilon^{n+1} = \varepsilon^n$, $\delta^{n+1} = \sigma\delta^n$}
		\ENDIF
		\STATE{\textbf{Step 3.} Set $n = n+1$ and go to Step 1.}
	\end{algorithmic}
\end{algorithm}

\noindent Algorithm \ref{algo:EXP_Algorithm_MINLPpaper} assumes that in Step 1 a so-called \textit{\dgopt}, i.e., an iterate fulfilling the condition in Step 1, can be found, for example via a global optimization method, see, e.g., \cite{locatelli2013global} for an overview of existing methods. Step 2 of the algorithm then provides a tool to decide when to increase penalization and when to seek for a better global minimizer.
\REV{Before we discuss the main convergence property of the algorithm,} we want to comment on the second condition in line $4$ of Algorithm \ref{algo:EXP_Algorithm_MINLPpaper}: this condition is based on \cite[\REV{Equation} $(3)$]{Lucidi_2011}, a H\"older-condition for the unpenalized objective function. Since our objective function $J$ is quadratic, it is H\"older-continuous with H\"older-exponent equal to $1$. Furthermore, the H\"older-constant that appears in the original formulation of the algorithm in \cite[Section \REV{3}]{Lucidi_2011}, can for simplicity be set to $1$ since it only influences the convergence speed of the algorithm. Thus, it does not appear in our formulation.
\REV{We notice that the updating rule in line $4$ of Algorithm \ref{algo:EXP_Algorithm_MINLPpaper} is tailored for problem \eqref{eq:MINLP_MINLPpaper}: the key step in here is choosing a suitable feasible point $z^n=[x^n]_{SR}$ with respect to which we can build up a neighborhood at minimum distance from $x^n$. This check is crucial to guarantee convergence of the given algorithmic scheme (see \cite[Lemma 1]{Lucidi_2011}). }

\REV{Before we prove a fundamental result in the upcoming Proposition \ref{prop:SR_MINLPpaper} that is the nedeed adaptation of \cite[Proposition 2]{Lucidi_2011}
, we give a useful definition. Once Proposition \ref{prop:SR_MINLPpaper} is proven, one can easily obtain \cite[Lemma 1]{Lucidi_2011} such that the convergence of Algorithm \ref{algo:EXP_Algorithm_MINLPpaper} follows from \REV{\cite[Theorem 2]{Lucidi_2011}} and \cite[Corollary $1$]{Lucidi_2011}.}

\begin{definition}
	The \textit{Chebyshev distance} between a point $x\in\R^{N+l}$ and a \REV{closed} set $C\subset\R^{N+l}$ is defined as 
	\begin{equation*}
	\Chdist(x,C)=\min_{y\in C}\norm{x-y}_\infty.
	\end{equation*}
\end{definition}

\begin{proposition}
	\label{prop:SR_MINLPpaper}
	Let $f$, $W$ and $X$ be the linear map and the sets defined in Lemma \ref{lemma:X_and_W_MINLPpaper}.
	For $z=(f(z_u),z_u)\transp\in W$\REV{, where $z_u$ denotes the control part of $z$}, let $B(z)$ be the set
	\begin{equation}\label{eq:-1}
	B(z) := \{x=(y,u)\transp\in\R^{N+l} \mid \norm{y}_\infty \leq \beta, \norm{u-z_u}_\infty \leq \rho\}.
	\end{equation}
	\REV{Letting $\rho < 0.5$ and choosing $\beta \geq \max_{z\in X} \norm{f(z_u)}_\infty$, it follows that  $z\in B(z)$ for all $z\in W$ and 
	$$
	B(z_a)\cap B(z_b)=\emptyset, \mbox{ for all } z_a,z_b \in W \mbox{ with } z_a\neq z_b.
	$$
	Furthermore,} given a point $\bar x = (f(\bar u),\bar u)\transp\in X$, then the point $\bar z := [\bar x]_{SR}\in W$ minimizes the Chebyshev distance between $\bar x$ and the sets $B(z)$ with $z\in W$, that is
	$$
	\bar z \in \argmin_{z\in W} \Chdist(\bar x,B(z)).
	$$
\end{proposition}

\begin{proof} 
\REV{With the choices of $\rho$ and $\beta$ the first statements are trivial. We furthermore mention that with $f$ being the linear map from Lemma \ref{lemma:X_and_W_MINLPpaper} and the $\norm{\cdot}_\infty$ inducing a matrix norm, we have
$$
\max_{z\in X} \norm{f(z_u)}_\infty = \max_{z\in X} \norm{K^{-1}M\Phi z_u}_\infty \leq \norm{K^{-1}M\Phi}_\infty \underbrace{\max_{z\in X} \norm{z_u}_\infty}_{=1}
$$
such that a finite $\beta$ can be chosen.
}
	
	Now, if there exists a $z\in W$ such that $\bar x \in B(z)$, it has to be $z \REV{=} \bar z = [\bar x]_{SR}$. In this trivial case, we have $\Chdist(\bar x, B(\bar z)) = 0$ and the result follows. 
	
	\REV{For the case that $\bar{x} \notin B(z)$ for any $z \in W$, we use a contradictory argument.}
	Therefore, we assume in the following that $\bar x\notin B(z)$ for all $z\in W$ \REV{and that there}  exists a point $\hat z = (f(\hat{z}_u),\hat{z}_u)\transp\in W$ satisfying
	\begin{equation}\label{eq:0}
	\Chdist(\bar x, B(\hat z)) < \Chdist(\bar x, B(\bar z)).    
	\end{equation}
	We can hence find two points $\hat p=(\hat p_y, \hat p_u)\transp \in B(\hat z)$ and $\bar p=(\bar p_y,\bar p_u)\transp\in B(\bar z)$ satisfying
	\begin{equation}\label{eq:1}
	\norm{\hat{p} - \bar{x}}_\infty = \Chdist(\bar{x},B(\hat{z}))\quad\mbox{and}\quad
	\norm{\bar{p} - \bar{x}}_\infty = \Chdist(\bar{x},B(\bar{z})).
	\end{equation}
	\REV{Equation \eqref{eq:1} together with the definition of the Chebyshev distance, the definition in \eqref{eq:-1}, as well as the choice of $\beta$ imply that $\hat p_y=\bar p_y=f(\bar u)$. Equation \eqref{eq:1} and the definition in \eqref{eq:-1} then furthermore yield }
	\begin{equation}\label{eq:2}
	\norm{\hat p_u-\hat z_u}_\infty=\norm{\bar p_u-\bar z_u}_\infty=\rho.
	\end{equation}
	\REV{We thus obtain}
	\begin{equation}
	\norm{\hat{p} - \bar{x}}_\infty = \max\{\norm{\hat{p}_u - \bar{u}}_\infty, \underbrace{\norm{\hat{p}_y - f(\bar{u})}_\infty}_{=0}\} = \norm{\hat{p}_u - \bar{u}}_\infty
	\end{equation}
	and equivalently $\norm{\bar{p} - \bar{x}}_\infty = \norm{\bar{p}_u - \bar{u}}_\infty$\REV{. As a consequence, the set defined in equation \eqref{eq:-1} together with the definition of the $\norm{\cdot}_\infty$--norm yield} 
	\begin{equation}\label{eq:3}
	\norm{\bar u-\hat
		z_u}_\infty=\norm{\bar u-\hat p_u}_\infty+\underbrace{\norm{\hat p_u-\hat z_u}_\infty}_{=\rho}
	\end{equation}
	\REV{as well as}
	\begin{equation}\label{eq:31}
	\norm{\bar u-\bar z_u}_\infty=\norm{\bar u-\bar p_u}_\infty+\underbrace{\norm{\bar p_u-\bar z_u}_\infty}_{=\rho},
	\end{equation}
	\REV{where the equalities stem from the fact that $\norm{u-z_u}_\infty \leq \rho$ inside \eqref{eq:-1} defines a unit cube.} 
	On the other hand, we obtain from \eqref{eq:0}  and \eqref{eq:1}
	that $\norm{\bar u-\bar p_u}_\infty > \norm{\bar u-\hat p_u}_\infty$ such that we conclude from equations \eqref{eq:0}-\eqref{eq:31} that
	\begin{equation}\label{ineq:fin}
	\norm{\bar u-\bar z_u}_\infty-\norm{\bar u-\hat z_u}_\infty=\norm{\bar u-\bar p_u}_\infty+\rho-\norm{\bar u-\hat p_u}_\infty-\rho>0.
	\end{equation}
	Remembering that $\bar z=[\bar x]_{SR}$, such that $\bar{z}_u \REV{=} \left[\bar u \right]_{SR}$, and that $\bar z\neq \hat z \, \Rightarrow \, \bar{z}_u \neq \hat{z}_u$ (follows from the definition of $W$), we know that for at least one component $i\in I=\{1,\dots, l\}$ it holds $\bar{z}_{u,i}\neq \hat{z}_{u,i}$. Let us now define the set 
	$$
	I_L := \{i\in I \mid \bar u_i\geq 0.5\}.
	$$
	If $|I_L|<S$, we have $\bar{z}_u = [\bar u]_{SR} = [\bar u]$, where $[\cdot]$ denotes the usual rounding \REV{to the nearest integer}, \REV{such that there exists a component $i\in I$ satisfying
$$
\|\bar{u} - \bar{z}_u\|_\infty = |\bar{u}_i - \bar{z}_{u,i}| \le 0.5 \le |\bar{u}_i - \hat{z}_{u,i}| \le
\|\bar{u} - \hat{z}_u\|_\infty
$$
}
	thus contradicting  \eqref{ineq:fin}.
	
	Therefore, we assume that $|I_L|\geq S$ in the following and define the set $I_S$, with $|I_S| = S$, so that $\bar u_i > \bar u_j$ for all $i\in I_S$ and \REV{all} $j\in I_L\setminus I_S$, i.e., the index set of the $S$ largest components of $\bar u$. By the definition of the smart rounding, it is then obvious that $\bar z_{u,i} = 1$ for $i\in I_S$ and $\bar z_{u,i} = 0$ for $i\in I\setminus I_S$.
	
	Now, any $\tilde z\in W$ can be obtained from $\bar z$ by considering any combination of the following operations:
	\begin{enumerate}
		\item $\bar z_{u,i}=1 \rightarrow \tilde z_{u,i}=0$ for one $i\in I_S$;
		\item $\bar z_{u,i}=1 \rightarrow \tilde z_{u,i}=0$ for one $i\in I_S$ and $\tilde z_{u,j}=1$ for one $j\in I\setminus I_L$;
		\item $\bar z_{u,i}=1 \rightarrow \tilde z_{u,i}=0$ for one $i\in I_S$ and $\tilde z_{u,j}=1$ for one $j\in I_L\setminus I_S$.
	\end{enumerate}
	Since $\bar u_i\geq 0.5$ for all $i\in I_S$, the first part of any of these operations results in
	$$
	|\bar u_i-\bar z_{u,i}|\leq  |\bar u_i-\tilde z_{u,i}|.
	$$
	In the second operation $j\in I\setminus I_L$ implies that $\bar u_j < 0.5$ and $\bar z_{u,j} = 0$ and we obtain
	$$
	|\bar u_j-\bar z_{u,j}|\leq  |\bar u_j-\tilde z_{u,j}|.
	$$
	In the third operation $j\in I_L\setminus I_S$ implies that $\bar u_j \geq 0.5$ but $\bar z_{u,j} = 0$ such that
	$$
	|\bar u_j-\bar z_{u,j}| \geq  |\bar u_j-\tilde z_{u,j}|.
	$$
	Taking the whole third operation into account and remembering that $i\in I_S$ as well as the definition of the smart rounding, we can see that
	$$
	\max\{|\bar u_j-\bar z_{u,j}|, |\bar u_i-\bar z_{u,i}|\} \leq \max\{|\bar u_j-\tilde z_{u,j}|, |\bar u_i-\tilde z_{u,i}|\}.
	$$
	Forming any $\tilde z\in W$ from $\bar z$ via these operations thus implies that
	$$
	\|\bar u-\bar z_u\|_\infty\leq \|\bar u-\tilde z_u\|_\infty
	$$
	and as especially $\hat z_u\in W$ can be obtained from $\bar z_u$, we have $\|\bar u-\bar z_u\|_\infty\leq \|\bar u-\hat z_u\|_\infty$ which is a contradiction to \eqref{ineq:fin}. Hence, we get that 
	$$
	\Chdist(\bar x, B(\hat z))\geq \Chdist(\bar x, B(\bar z)),\mbox{ for all }\hat z\in W,
	$$
	which concludes the proof.
\end{proof}

\noindent\REV{We end this section with the} main convergence property of Algorithm \ref{algo:EXP_Algorithm_MINLPpaper} \REV{that} is reported in the upcoming Proposition
\ref{prop:exp}\REV{. It} shows that Algorithm \ref{algo:EXP_Algorithm_MINLPpaper} extends global optimization methods for continuous problems to integer problems.

\begin{proposition}
	\label{prop:exp}
	Every accumulation point $x^*$ of a sequence of iterates $\{x^n\}_{n\in\N}$ of Algorithm \ref{algo:EXP_Algorithm_MINLPpaper} is a global minimizer of \eqref{eq:MINLP_MINLPpaper}.
\end{proposition}
\begin{proof}
	Using Proposition \ref{prop:SR_MINLPpaper}, \REV{we can easily get \cite[Lemma 1]{Lucidi_2011} proven. Hence,}   the statement follows from \REV{\cite[Theorem 2]{Lucidi_2011}} and \cite[Corollary $1$]{Lucidi_2011}.
\end{proof}

\subsection{Development of the improved penalty algorithm (IPA)}
\label{sec:3_2_MINLPpaper}

We want to stress that the EXP algorithm considered in the previous section needs to calculate a \dgopt in Step 1 of each iteration. As local minima are introduced around integer points in \eqref{eq:MINLP_penalty} for sufficiently large values of $\varepsilon$, finding a \dgopt requires the use of a global deterministic continuous optimization solver. Thus, the EXP algorithm, albeit  providing a theoretical framework for when to increase the amount of penalization and when to search for a better minimizer, might be too costly, especially in a large-scale MIPDECO context. 

This is our motivation to drop the requirement for a \dgopt in Step 1 of Algorithm \ref{algo:EXP_Algorithm_MINLPpaper} and compute an iterate $x^n\in X$ that simply reduces the objective function (i.e., $\Jt(x^n;\varepsilon^n) < \Jt(x^{n-1};\varepsilon^n)$). In order to do so, we employ a probabilistic approach that, in each iteration, aims at improving the current iterate by perturbing it and utilizing this perturbation as initial guess for a tailored local optimization solver (see Section \ref{sec:4_2_MINLPpaper} for a detailed description of the solver). Do note that this strategy is closely connected to classic basin hopping or iterated local search strategies, see, e.g.,  \cite{grosso2007population,leary2000global}, for global optimization problems.
By combining these two ideas, we end up with  Sub-Algorithm \ref{algo:Perturbation} that is then invoked in Step 1 of the novel \textit{improved penalty algorithm (IPA)}, i.e., \REV{the combination of the Algorithms \ref{algo:some_Bologna_Algorithm} \& \ref{algo:Perturbation}} reported below.



\begin{algorithm}[ht]
	\caption{Improved penalty algorithm($x^0\in X$, $\varepsilon^0 >0$, $\sigma\in (0,1)$, \REV{$\pmax\in\N$})}
	\begin{algorithmic}[1]
		\label{algo:some_Bologna_Algorithm}
		\STATE{$n = 0$, $x^n = x^0$, $\varepsilon^n = \varepsilon^0$}
		\STATE{\textbf{Step 1.} Call Algorithm \ref{algo:Perturbation}($x^n$, \REV{$\pmax$}, $\varepsilon^n$) to generate a new iterate $x^{n+1}$}.
		\STATE{\textbf{Step 2.}}
		\IF{$x^{n+1}\notin W$ \textbf{and} $\Jt(x^{n+1};\varepsilon^n) - \Jt([x^{n+1}]_{SR};\varepsilon^n) \leq \varepsilon^n \norm{x^{n+1} - [x^{n+1}]_{SR}}_2$}
		\STATE{$\varepsilon^{n+1} = \sigma\varepsilon^n$}
		\ELSE
		\STATE{$\varepsilon^{n+1} = \varepsilon^n$}
		\ENDIF
		\STATE{\textbf{Step 3.}}
		\IF{$x^n \REV{=} x^{n+1}$}
		\RETURN{$[x^{n+1}]_{SR}$}
		\ELSE
		\STATE{Set $n = n+1$ and go to Step 1.}
		\ENDIF
	\end{algorithmic}
\end{algorithm}

\begin{algorithm}[ht]
	\renewcommand{\thealgorithm}{\ref{algo:some_Bologna_Algorithm}.a} 
	\caption{Reduction via perturbation($x\in X$, \REV{$\pmax\in\N$}, $\varepsilon >0$)}
	\begin{algorithmic}[1]
		\label{algo:Perturbation}
		\STATE{$x^{init} = x$}
		\FOR{$j=1,\dots, \pmax$}
		\STATE{Use a local optimization solver to determine a solution $x^{loc}$ of \eqref{eq:MINLP_penalty} for $\varepsilon$ using $x^{init}$ as initial guess.}
		\IF{$\Jt(x^{loc};\varepsilon) < \Jt(x;\varepsilon)$}
		\RETURN{$x^{loc}$}
		\ELSE
		\STATE{Generate a point $x^{pert} = Perturbation(x^{loc})$  and set $x^{init} = x^{pert}$.}
		\ENDIF
		\ENDFOR
		\RETURN{$x$}
	\end{algorithmic}
\end{algorithm}

\begin{remark} Clearly, there is a trade-off between the two algorithms: while Algorithm \ref{algo:EXP_Algorithm_MINLPpaper} guarantees deterministic convergence, it is unable to tackle large-scale problems. This downside is then lifted in  Algorithm \ref{algo:some_Bologna_Algorithm} thanks to the combination of a local solver and a probabilistic global search approach. \REV{Here, the probabilistic nature of Algorithm \ref{algo:some_Bologna_Algorithm} lies in the perturbation operation in line $7$ of Algorithm \ref{algo:Perturbation} and we will go into detail in Section \ref{sec:4_1_MINLPpaper}.}
As a result, no deterministic convergence property is available for the overall IPA. Nevertheless, the framework underlying the EXP algorithm of when to increase the penalization and when to look for a better minimizer still supports the novel Algorithm \ref{algo:some_Bologna_Algorithm}.
\end{remark}

\noindent Do note that if the for-loop in Algorithm \ref{algo:Perturbation} \REV{reaches the iteration limit} (and thus no better iterate was found after $\pmax\REV{\in\N}$ perturbations), the algorithm terminates with $x$, which was the input iterate. In that case it is $x^{n+1} \REV{=} x^n$ and the overall Algorithm \ref{algo:some_Bologna_Algorithm} then terminates. Therefore, the perturbation strategy in Algorithm \ref{algo:Perturbation} together with the choice of $\pmax$ give the information at what point no further reduction in the objective function can be found. Algorithm \ref{algo:Perturbation} does not specify a perturbation strategy in line $7$ and one can develop a strategy that \REV{is beneficial for one's model problem.} We will present our strategy in the upcoming Section \ref{sec:4_1_MINLPpaper}.

While $\varepsilon$ is decreased during Algorithm \ref{algo:some_Bologna_Algorithm} (and thus the amount of penalization is increased), \REV{the concave penalty term \eqref{eq:penaltyterm_MINLPpaper} grows larger and introduces local minima to the objective function $\Jt(x;\varepsilon)$. As $\varepsilon$ is further decreased, the shape of the objective function continues to change such that these local minima then move towards the integer points (follows from the definition of the penalty term).}
Due to this behavior, the condition $\Jt(x^{loc};\varepsilon^n) < \Jt(x;\varepsilon^n)$ in line $4$ of Algorithm \ref{algo:Perturbation} is always fulfilled as long as $\varepsilon^n < \varepsilon^{n-1}$ holds in Algorithm \ref{algo:some_Bologna_Algorithm}. 
\REV{Thus, as long as it was $\varepsilon^n < \varepsilon^{n-1}$ in Algorithm \ref{algo:some_Bologna_Algorithm}, the for-loop in Algorithm \ref{algo:Perturbation} terminates after the first iteration and the perturbation loop is not invoked.} 
\REV{As a result}, we expect Algorithm \ref{algo:some_Bologna_Algorithm} to have a two-phase behavior: in the first phase, the penalization is increased \REV{due to line $5$ of Algorithm \ref{algo:some_Bologna_Algorithm}} until a feasible \REV{integer} iterate $x^{n+1}\in W$ is found \REV{(due to the nature of the penalty term this has to happen for small enough values of $\varepsilon$)}.
\REV{Do note that, as long as the current iterate is not too close to an integer, the second condition of line $4$ is also fulfilled (due to the shape of the objective function the left hand side is then negative).}
\REV{If an iterate $x^{n+1}\in W$ is found (or one is close enough such that the second condition on line $4$ is violated), line $7$ of Algorithm \ref{algo:some_Bologna_Algorithm} keeps the amount of penalization such that the shape of the objective function remains the same and a better iterate can only be found via the perturbation strategy inside Algorithm \ref{algo:Perturbation} in Step 1 of Algorithm \ref{algo:some_Bologna_Algorithm} (after line $7$ of Algorithm \ref{algo:some_Bologna_Algorithm}, Algorithm \ref{algo:Perturbation} is called with $\varepsilon^n = \varepsilon^{n-1}$ such that the first iteration of the for loop simply reproduces the local minimum which was the input iterate).}
\REV{Thus}, Algorithm \ref{algo:Perturbation} then tries to improve the current iterate by perturbing it and restarting the local solver with this perturbed iterate. This way, one wants to escape bad basins of attraction of $\Jt$ and then move towards better local \REV{minimizers} and eventually the global one.


Finally, we want to mention that a new iterate $x^{n+1} = \left[y^{n+1},u^{n+1}\right]\transp$ found by Algorithm \ref{algo:Perturbation} is always feasible so that $x^{n+1}\in X$. Thus, the criterion $x^{n+1}\notin W$ in line $4$ of Algorithm \ref{algo:some_Bologna_Algorithm} can, in an actual implementation, be replaced by 
$$\norm{u^{n+1} - [u^{n+1}]_{SR}}_\infty > \efeas$$
with some feasibility tolerance $\efeas$. Thus, it is reasonable to return $[x^{n+1}]_{SR}$ such that the control of our output iterate is always integer and respects the knapsack constraint.

\section{Algorithmic details, local solver, and numerical setup}
\label{sec:4_MINLPpaper}

We begin with a discussion on various details of our implementation of the IPA 
including the perturbation strategy and the local solver. Afterwards, we shortly introduce two other solution strategies for  \eqref{eq:MINLP_MINLPpaper} and discuss the setup for the numerical investigation that will be carried out in Section \ref{sec:5_MINLPpaper}.

\subsection{Implementation details of the IPA}
\label{sec:4_1_MINLPpaper}

We start with the presentation of our perturbation strategy used in Algorithm \ref{algo:Perturbation}. The details are described in Algorithm \ref{algo:flipping} \REV{and since the algorithm has $\theta\in\N$ as an input, this is consequently another input of the overall IPA.}

\begin{algorithm}[ht]
	\renewcommand{\thealgorithm}{\ref{algo:some_Bologna_Algorithm}.b} 
	\caption{Perturbation($x\in X$\REV{, $\theta\in\N$})}
	\begin{algorithmic}[1]
		\label{algo:flipping}
		\STATE{Split $x=(y,u)\transp$ into the state $y\in\R^N$ and control $u\in\R^l$. Define $u^{pert} := u$.}
		\STATE{Find $I_{S}$, the set containing the indices of the entries of $u$ that are larger than $\frac{1}{2}$.}
		\FOR{$j=1,\dots, \min\{\vert I_S\vert,\theta\}$}
		\STATE{Randomly select $\hat{\imath}\in I_{S}$.}
		\STATE{Define $I_{adj}$ the set of indices corresponding to sources \textit{adjacent} to $\tilde{x}_{\hat{\imath}}$.}
		\STATE{Randomly select $\hat{\imath}_{adj}\in I_{adj}$.}
		\STATE{Set $\left(u^{pert}\right)_{\hat{\imath}}$ to a randomly chosen value smaller than $\frac{1}{2}$.}
		\STATE{Set $\left(u^{pert}\right)_{\hat{\imath}_{adj}}$ to a randomly chosen value larger than $\frac{1}{2}$.}
		\STATE{Remove $\hat{\imath}$ from $I_{S}$.}
		\ENDFOR
		\STATE{Compute the state $y^{pert}$ corresponding to $u^{pert}$, i.e., $y^{pert} = f(u^{pert})$.}
		\RETURN{$x^{pert} := \left[y^{pert},u^{pert}\right]\transp$}
	\end{algorithmic}
\end{algorithm}

\noindent As mentioned in Section \ref{sec:3_2_MINLPpaper}, this perturbation strategy should only be called upon in the later stages of Algorithm \ref{algo:some_Bologna_Algorithm} where the amount of penalization is significant enough such that the set $I_S$ in Algorithm \ref{algo:flipping} is not empty.
When Algorithm \ref{algo:flipping} is called by Algorithm \ref{algo:Perturbation} inside Algorithm \ref{algo:some_Bologna_Algorithm}, $x$ is equal to the current iterate $x^n = \left[y^{n},u^{n}\right]\transp$. The algorithm then essentially performs $\theta\in\N$ \textit{flips} to the current control $u^n$, where a flip is one iteration of the for-loop of Algorithm \ref{algo:flipping}, i.e., a large value of $u^n$ is set to a small value and an entry of $u^n$ corresponding to a source that is \textit{adjacent} to the source corresponding to the large value is set to a large value. By this strategy the resulting perturbation $x^{pert}$ possibly lies outside the current basin of attraction and therefore might be an initial guess for the local solver in Algorithm \ref{algo:Perturbation} resulting in a point with a potentially better function value. It remains to explain what we mean by \textit{adjacent} in the above context.
\REV{
\begin{definition}
\label{def:adjind_MINLPpaper}    
Given a collection of points $x_1,\dots ,x_n \in \Omega$ and a radius $r > 0$, we define for a point $x_i$ the set of \textit{adjacent indices}
$$
I_{adj} := \{j\in \{1,\dots, n\} \mid j\neq i,~ \norm{x_i - x_j}_\infty \leq r\}.
$$
\end{definition}
}

\noindent\REV{Inside Algorithm \ref{algo:flipping}, we can thus obtain $I_{adj}$ via Definition \ref{def:adjind_MINLPpaper} where we use the centers $\tilde{x}_1,\dots, \tilde{x}_l\in \Omega$ of our source functions as points. Assuming that they are arranged in a uniform $m\times m$ grid, a possible radius might be $r = \frac{1}{m}$.}  


Although the perturbation strategy presented so far depends on the uniform grid of source centers in order to determine the index set $I_{adj}$, the underlying concept of this \textit{flipping} does not depend on the chosen modelling. The large component $\left(u^{pert}\right)_{\hat{\imath}}$ of the control can often be associated to a spatial counterpart denoted, for the purpose of clarity, as $x_{\hat{\imath}}$ here. In our case this is $x_{\hat{\imath}} = \tilde{x}_{\hat{\imath}}$, the center of the Gaussian source function. If the control would for example be modeled via piecewise constant functions $\{\chi_i(x)\}_{i=1}^l$ (as in \cite{Buchheim_2018} or Section \ref{sec:5_2_MINLPpaper}), $x_{\hat{\imath}}$ could be the center of the patch of the subdomain that corresponds to $\chi_{\hat{\imath}}(x)$. If the control would be distributed, $x_{\hat{\imath}}$ would be the vertex of the grid that corresponds to $u_{\hat{\imath}}$. 
\REV{Thus, one would always find a set of points for Definition \ref{def:adjind_MINLPpaper} and could then select a proper radius (and of course a suitable norm).}
With this interpretation, as long as the control can be associated to spatial counterparts of the domain $\Omega$, the presented perturbation strategy can easily be applied to different kinds of controls, models, and domains.

Finally, we found it effective in our experiments to set $\left(u^{pert}\right)_{\hat{\imath}}$ to a random value in $[0.1,0.2]$ during Algorithm \ref{algo:flipping}. Afterwards, we calculate $d_{\hat{\imath}} := \vert \left(u\right)_{\hat{\imath}} - \left(u^{pert}\right)_{\hat{\imath}} \vert$ and set $\left(u^{pert}\right)_{\hat{\imath}_{adj}}$ to a random value in $[d_{\hat{\imath}}-0.1,d_{\hat{\imath}}]$. This strategy ensures that the perturbed control $u^{pert}$ is still feasible (especially fulfilling the knapsack constraint). Furthermore, this prohibits the perturbed control of having values that are too close to $0$ or $1$. By this, $x^{pert}$ is then an initial guess for the local solver in Algorithm \ref{algo:Perturbation} that (possibly) lies outside the current basin of attraction and is at the same time not too close to other local minimizers (at this stage of the IPA there are local minimizers nearby all integer points).

In the remainder of this section, we want to discuss the termination of Algorithm \ref{algo:Perturbation} and thus Algorithm \ref{algo:some_Bologna_Algorithm}.
The criterion $\Jt(x^{loc};\varepsilon^n) < \Jt(x^n;\varepsilon^n)$ in Algorithm \ref{algo:Perturbation} (as it is called inside Algorithm \ref{algo:some_Bologna_Algorithm} with $x=x^n$ and $\varepsilon = \varepsilon^n$) can be numerically challenging in an actual implementation. Although the criterion should be fulfilled when it was $\varepsilon^n < \varepsilon^{n-1}$ in Algorithm \ref{algo:some_Bologna_Algorithm} as mentioned in Section \ref{sec:3_2_MINLPpaper}, this might not be the case numerically since any local solver used in Algorithm \ref{algo:Perturbation} only computes $x^{loc} = \left[y^{loc}, u^{loc}\right]\transp$ up to some internal tolerance. Furthermore, if $x^n$ is close to an integer already, we do not want to accidentally fulfill $\Jt(x^{loc};\varepsilon^n) < \Jt(x^n;\varepsilon^n)$ due to numerical effects although $[u^{loc}] \REV{=} [u^n]$ such that no progress towards a better integer solution would be made. To cover both of these cases in our implementation, we first calculate the two distances 
$$
d^{loc} := \norm{u^{loc} - u^n}_\infty\quad\mbox{and}\quad d_{SR}^{loc} := \norm{[u^{loc}]_{SR} - [u^n]_{SR}}_\infty
$$ 
and replace $\Jt(x^{loc};\varepsilon^n) < \Jt(x^n;\varepsilon^n)$ by the following two criteria and thus return $x^{loc}$ in Algorithm \ref{algo:Perturbation} if one of these is fulfilled.
\begin{enumerate}
	\item If $\varepsilon^n < \varepsilon^{n-1}$ and either $\Jt(x^{loc};\varepsilon^n) < \Jt(x^n;\varepsilon^n)$, $d^{loc} < 0.2$, or $d_{SR}^{loc} = 0$ are fulfilled.
	\item If $\varepsilon^n = \varepsilon^{n-1}$ and $d_{SR}^{loc} \neq 0$, as well as $\Jt(x^{loc};\varepsilon^n) < \Jt(x^n;\varepsilon^n)$, and additionally  $\Jt([x^{loc}]_{SR};\varepsilon^n) < \Jt([x^n]_{SR};\varepsilon^n)$ are fulfilled.
\end{enumerate}
The first criterion targets the case when $\Jt(x^{loc};\varepsilon^n) \nless \Jt(x^n;\varepsilon^n)$ numerically (although it was $\varepsilon^n < \varepsilon^{n-1}$ in Algorithm \ref{algo:some_Bologna_Algorithm}) and thus also accepts iterates that are either close to, or presumably in the same basin of attraction as, the previous iterate. We mention that this usually happens during the first phase of the IPA where the amount of penalization is increased (due to $\varepsilon^{n+1} = \sigma\varepsilon^n$) and is not yet large enough for the local solver to produce near integer solutions fulfilling $\norm{u^{loc} - [u^{loc}]_{SR}}_\infty \leq \efeas$.
As a result it is not necessary to search for better solutions via the perturbation strategy such that this criterion tries to prevent non-productive iterations in Algorithm \ref{algo:Perturbation}.
If a feasible iterate was found and the amount of penalization was not increased, the second criterion only accepts better iterates that lie outside the current basin of attraction and thus enforces progress towards a better integer solution and should prevent the algorithm from getting stuck in an unsatisfactory local minimum.

\subsection{Implementation of the local solver: an interior point framework for the large-scale setting}
\label{sec:4_2_MINLPpaper}

We now discuss our implementation of line $3$ in Algorithm \ref{algo:Perturbation}, that is the choice of the local solver for finding a solution $x^{loc}$ of \eqref{eq:MINLP_penalty} for a given $\varepsilon$.
Due to the structure of \eqref{eq:MINLP_penalty}, which was equivalent to \eqref{eq:MINLP_raw_penalty_MINLPpaper}, we opt for an interior point method (IPM) that is particularly suitable for solving quadratic programming problems and it also allows the use of an efficient preconditioner in the linear algebra phase, see, e.g., \cite{nocedalwright1999numericalopt,Gondzio_2012}. 
Following \cite{Gondzio_2012}, we present the derivation of a standard interior point method for the following reformulation of problem \eqref{eq:MINLP_penalty}, that is 
\begin{equation*}
\begin{array}{cl}\label{pb_pe2}
\displaystyle\min_{y\in \R^N,u\in \R^l, z\in \R} &  \Jt(y,u;\varepsilon) = \frac 1 2 (y-y_d)\transp M (y-y_d)  + \frac{1}{\varepsilon} ({\bf 1}\transp u-u\transp u), \\
\mbox{ s.t. } 
& Ky = M\Phi u \qquad \mbox{ and } \qquad  {\bf 1}\transp u + z - S =0,\\
&   \\
& 0 \le u \le 1 \qquad \mbox{ and } \qquad z \ge 0,
\end{array}
\end{equation*}
where $z\geq 0$ is a scalar slack variable and the notation has been adapted to distinguish the control $u$ and the state $y$. For the sake of generality we include the case when the stiffness matrix $K$ is  non-symmetric.
The main idea of an IPM is the elimination of the inequality constraints on $u$ and $z$ via the introduction of corresponding logarithmic barrier functions. The Lagrangian associated with the barrier subproblem reads
\begin{align*}
L_{\mu,\varepsilon}(y,u,z; p,q)  = &\Jt(y,u;\varepsilon) + p\transp (Ky - M\Phi u) + q ({\bf 1}\transp u + z - S)\\
&-\mu \sum_{i=1}^l \log(u_i) - \mu \sum_{i=1}^l \log(1 - u_i) - \mu \log (z), 
\end{align*}
where $p\in\R^N$ is the Lagrange multiplier (or adjoint variable) associated with the state equation, $q \in \R $ is the Lagrange multiplier associated with the scalar equation ${\bf 1}\transp u + z - S =0$, and $\mu > 0$ is the barrier parameter that controls the relation between the barrier term and the original objective $\Jt(y,u;\varepsilon)$. As the method progresses,  $\mu$ is decreased towards zero.

First-order optimality conditions are derived by applying duality theory resulting in a nonlinear system parametrized by $\mu$ as detailed below. Thus, differentiating $L_{\mu,\varepsilon}$ with respect to the variables $y$, $u$, $z$, $p$, and $q$ gives the nonlinear system
\begin{subequations}
	\label{eq:kkt}
	\begin{eqnarray}
	M y - M y_d + K\transp  p  &= 0, \\ 
	\frac{1}{\varepsilon}({\bf 1} - 2u) - \Phi\transp M p + q {\bf 1} - \lambda_{u,0} + \lambda_{u,1} &= 0,\\ 
	q - \lambda_{z,0}  = 0, 
	\quad Ky - M\Phi u = 0,\quad {\bf 1}\transp u + z - S &= 0, 
	\end{eqnarray}
\end{subequations}
where the Lagrange multipliers $\lambda_{u,0}$, $\lambda_{u,1}\in\R^l$ and $\lambda_{z,0}\in\R$ are defined as
\begin{equation}\label{lagmul}
(\lambda_{u,0})_i := \frac{\mu}{u_i},~ (\lambda_{u,1})_{i} := \frac{\mu}{1 - u_i}, \quad\mbox{for } i = 1,\dots , l,\quad  \mbox{and}\quad
\lambda_{z,0} := \frac{\mu}{z}.
\end{equation}
Furthermore, the bound constraints $\lambda_{u,0} \geq 0$, $\lambda_{u,1} \geq 0$, and $\lambda_{z,0} \geq 0$ then enforce the constraints on $u$ and $z$.

The crucial step of deriving the IPM is the application of Newton's method to the above nonlinear system. Letting $y$, $u$, $z$, $p$, $q$, $\lambda_{u,0}$, $\lambda_{u,1}$, and $\lambda_{z,0}$ denote the most recent Newton iterates, these are then updated in each iteration by computing the corresponding Newton steps $\Delta y$, $\Delta u$, $\Delta z$, $\Delta p$, $\Delta q$, $\Delta \lambda_{u,0}$, $\Delta \lambda_{u,1}$, and $\Delta \lambda_{z,0}$ through the solution of the following Newton system
\begin{align}
\label{eq:NewtonSystem}
\begin{bmatrix}
M            &    0     &   0   &   K\transp   & 0  \\
0             &  -\frac 2 \varepsilon I_l + \Theta_u &   0 & -\Phi\transp M  & {\bf 1} \\
0             &  0                       &   \theta_z   & 0           &  1    \\
K          &  -M\Phi  &   0   &  0          &  0    \\
0            &  {\bf 1}\transp       &   1  &   0     & 0
\end{bmatrix}
&\begin{bmatrix}
\Delta y\\
\Delta u \\
\Delta z \\
\Delta p \\
\Delta q
\end{bmatrix} \\
\nonumber &\hspace*{-4em}=
-\begin{bmatrix}
M y - M y_d + K\transp  p \\
\frac{1}{\varepsilon}({\bf 1} - 2u) -\Phi\transp M  p + q {\bf 1} - \lambda_{u,0} + \lambda_{u,1} \\		
q - \lambda_{z,0} \\ 
Ky-M\Phi u\\
{\bf 1}\transp u + z - S
\end{bmatrix}.
\end{align}
Here, $\Theta_u := U^{-1} \Lambda_{u,0} + (I_l - U )^{-1} \Lambda_{u,1}$, $\theta_z := \lambda_{z,0}/z$, and $U,~\Lambda_{u,0}$, and $\Lambda_{u,1}$ are diagonal matrices with the most recent iterates of $u$, $\lambda_{u,0}$, and $\lambda_{u,1}$ appearing on their diagonal entries.
The matrices $\Theta_u$ and $\theta_z >0$, while being positive definite, are typically very ill-conditioned.
Also, due to the term $-\frac 2 \varepsilon I_l $, the block  $-\frac 2 \varepsilon I_l + \Theta_u$
may be indefinite, especially for small values of $\varepsilon$.
Following suggestions in \cite[Chapter 19.3]{nocedalwright1999numericalopt} to handle
nonconvexities in the objective function by promoting the computation of descent directions, we heuristically keep the diagonal matrix $-\frac 2 \varepsilon I_l + \Theta_u$ positive definite by setting any negative values to some value $\gamma > 0$.
Once the above system is solved, one can compute the steps for the Lagrange multipliers via
\begin{align*}
\Delta \lambda_{u,0} & = - U^{-1} \Lambda_{u,0} \Delta u - \lambda_{u,0} + \mu U^{-1}{\bf 1},\\
\Delta \lambda_{u,1} & = (I_l-U)^{-1} \Lambda_{u,1} \Delta u - \lambda_{u,1} + \mu (I_l-U)^{-1}{\bf 1},\\
\Delta \lambda_{z,0} & = - (\lambda_{z,0}/z) \Delta z - \lambda_{z,0} + \mu/z.
\end{align*}
A general IPM implementation only involves one Newton step per iteration. Thus, after choosing suitable step-lengths so that the updated iterates remain feasible, the new iterates can be calculated and the barrier parameter $\mu$ is reduced, thus concluding one iteration of the IPM.
Finally, we report the primal and dual feasibilities 
\begin{equation*}
\label{eq:prdufeas_MINLPpaper}
\xi_p:=  \begin{bmatrix} Ky - M\Phi u \, \\
{\bf 1}\transp u+ z- S   
\end{bmatrix}
\quad \mbox{and}\quad 
\xi_d:= \begin{bmatrix}
M y- M y_d + K\transp  p   \\
\frac 1 \varepsilon ({\bf 1} - 2 u) -\Phi\transp M p+ q{\bf 1} - \lambda_{u,0} + \lambda_{u,1} \\                
q- \lambda_{z,0}  
\end{bmatrix},
\end{equation*}
as well as the complementarity gap
\begin{equation*}
\xi_c :=
\begin{bmatrix}
U\lambda_{u,0} - \mu{\bf 1}, & (I_l - U) \lambda_{u,1} - \mu{\bf 1}, & z\lambda_{z,0} - \mu
\end{bmatrix}\transp,
\end{equation*}
as measuring the change in the norms of $\xi_p$, $\xi_d$ and $\xi_c$ allows us to monitor the convergence of the entire process. 

Clearly, the computational burden of this IPM lies in the solution of the Newton system \eqref{eq:NewtonSystem} and our strategy regarding this issue is twofold: on the one hand we employ an inexact Newton-Krylov strategy for the solution of the nonlinear system \eqref{eq:kkt} and on the other hand we design a suitable preconditioner to speed up the convergence of our Krylov method of choice for the Newton system \eqref{eq:NewtonSystem}.
Regarding the inexactness strategy, the idea is to increase the accuracy in the solution of the
Newton equation as $\mu$ decreases. This minimizes the occurrence of so-called \textit{oversolving} in the first interior point steps.
Global convergence results to a solution of the first-order optimality conditions for the resulting inexact IPM can be found in \cite{bellavia1998inexact}. 


\paragraph{Preconditioning for the interior point method}
We will now present our linear algebra strategy for the solution of the Newton system \eqref{eq:NewtonSystem}, i.e., we choose our Krylov method and design a suitable preconditioner.
Investigating the system matrix of the Newton system, we observe that with the choice
$$
A=
\begin{bmatrix}
M            &    0     &   0   \\
0             &  -\frac 2 \varepsilon I_l + \Theta_u &   0 \\
0             &  0                       &   \theta_z     \\
\end{bmatrix},\qquad
B=
\begin{bmatrix}
K          &  -M\Phi  &   0     \\
0            &  {\bf 1}\transp       &   1   \\              
\end{bmatrix}
$$
we have to solve a saddle point system
$
\begin{bmatrix}
A&B^T\\
B&0
\end{bmatrix}.
$
As discussed already, the block 
$-\frac 2 \varepsilon I_l + \Theta_u$ is kept positive definite throughout the interior point method, so that we can assume that $A$ is positive definite.

Such systems are a cornerstone of applied mathematics and appear in many application scenarios, see, e.g., \cite{book::andy,BenGolLie05}. While the system is symmetric and we could apply \minres \cite{minres}, we here use a nonsymmetric method, namely \gmres \cite{gmres}, because we found that a block-triangular preconditioner 
$
\begin{bmatrix}
\hat{A}&0\\
B&-\hat{S}
\end{bmatrix}
$
performs better in our experiments. It would also be possible to use symmetric solvers, which are based on nonstandard inner products, see, e.g., \cite{stoll::bp+comb,dgsw08}. 
We here focus on the design of the approximations for the $(1,1)$-block $\hat{A}\approx A$ and for the Schur-complement
$$
\hat{S}\approx S=
\begin{bmatrix}
K M^{-1} K\transp + M\Phi (-\frac 2 \varepsilon I_l + \Theta_u)^{-1}\Phi\transp M &   0     \\
0  &  {\bf 1}\transp(-\frac 2 \varepsilon I_l + \Theta_u)^{-1}{\bf 1}+\theta_z^{-1}   \\  
\end{bmatrix}.
$$
In our preconditioning approach, we neglect the term ${\bf 1}\transp(-\frac 2 \varepsilon I_l + \Theta_u)^{-1}{\bf 1}+\theta_z^{-1} $ and set the preconditioner to $1$ as this typically does not result in many additional iterations and we avoid dealing with the ill-conditioning of both $(-\frac 2 \varepsilon I_l + \Theta_u)$ and $\theta_z$. In our setup here we thus end up with the following approximation
$$
\hat{S}=
\begin{bmatrix}
K M^{-1} K\transp  &  0\\
0 & 1\\  
\end{bmatrix}\quad\mbox{and}\quad A=\hat{A}.
$$
We close this section with two short remarks.

\begin{remark}
	The purpose of this basic preconditioner is to speed up the solution process of our IPM, but for future research we need to enhance this based on recent progresses in preconditioning for interior point methods, see, e.g., \cite{pearson2020interior,pearson2019block,bergamaschi2004preconditioning}.
\end{remark}
\begin{remark}
	\label{rem:nonlinIPM_MINLPpaper}
	Although this IPM together with the preconditioner are formulated for the penalty formulation \eqref{eq:MINLP_penalty} that refers to the model problem \eqref{eq:PDE_MINLPpaper}, it is clear that the IPM generalizes to general linear PDE constraints (in fact, Section \ref{sec:5_2_MINLPpaper} will contain experiments for a convection-diffusion problem resulting in a nonsymmetric stiffness matrix $K$). Furthermore, the IPM and the preconditioner can be formally adapted to a nonlinear PDE constraint $F(y,u) = 0$, where $F:\R ^{N+l}\rightarrow \R^N$ is a smooth nonlinear function. One simply has to introduce $F'(y,u) \in \R^{N\times {(N+l)}}$, the Jacobian of $F$ as well as $F'_y \in \R^{N\times N}$ and $F'_u \in \R^{N\times l}$, the submatrices of the Jacobian such that $F'(y,u) = [F'_y,\, F'_u]$. We then obtain the IPM for this nonlinear problem by replacing in the Newton system \eqref{eq:NewtonSystem} $K\transp$ by $(F'_y)\transp $, $-M\Phi$ by $F'_u$ (and thus $-\Phi\transp M$ by $(F'_u)\transp$), and $Ky-M\Phi u$ by $F(y,u)$. In the nonlinear case, convergence of the IPM is ensured when embedded in suitable globalization strategies \cite{nocedalwright1999numericalopt}.
\end{remark}

\subsection{Simple penalty and \BnB method}
\label{sec:4_3_MINLPpaper}
We shortly discuss a simple penalty approach and our \BnB solver of choice for the solution of \eqref{eq:MINLP_MINLPpaper} to which we want to compare our IPA in the numerical Section \ref{sec:5_MINLPpaper}.

Starting with the penalty formulation \eqref{eq:MINLP_penalty}, we follow the simple iterative penalization strategy already mentioned in the introduction: in each iteration, use a local solver to determine a solution of \eqref{eq:MINLP_penalty} which then is the next iterate; decrease the penalty parameter; stop if the control part of the new iterate is integer. This approach leads to the following \textit{penalty algorithm}, i.e., Algorithm \ref{algo:some_Algorithm}.


\addtocounter{algorithm}{-2} 
\begin{algorithm}
	\caption{Penalty($x^{0}\in X$, $\varepsilon^0>0$, $\sigma\in (0,1)$)}
	\begin{algorithmic}[1]
		\label{algo:some_Algorithm}
		\STATE{$n = 0$, $x^n = x^0$, $\varepsilon^n = \varepsilon^0$}
		\REPEAT
		\STATE{Use a local solver to determine a solution $x^{n+1}$ of \eqref{eq:MINLP_penalty} for $\varepsilon^n$ using $x^{n}$ as initial guess.}
		\STATE{$\varepsilon^{n+1} = \sigma\varepsilon^n$}
		\STATE{$n=n+1$}
		\UNTIL{$\norm{u^{n} - [u^{n}]_{SR}}_\infty < \efeas$}
		\RETURN{$[x^n]_{SR}$}
	\end{algorithmic}
\end{algorithm}

\noindent As already discussed in Section \ref{sec:3_2_MINLPpaper}, we use the criterion $\norm{u^{n} - [u^{n}]_{SR}}_\infty < \efeas$ instead of $x^n\in W$ to determine whether or not an integer iterate has been found and then return $[x^n]_{SR}$.
Algorithm \ref{algo:some_Algorithm} is a simplification of the IPA in several ways: the penalty parameter is reduced in every iteration, a new iterate $x^{n+1}$ generated by the local solver is always accepted as such, and the algorithm terminates as soon as an iterate $x^{n+1}\in W$ is found. 
Thus, Algorithm \ref{algo:some_Algorithm} has no theoretical justification, whereas Algorithm \ref{algo:some_Bologna_Algorithm} utilizes the theoretical framework of the EXP algorithm for the correct selection of the penalty parameter as well as a local iterative search strategy for the computation of the new iterate.


Finally, we choose \cplexmiqp the \BnB routine of CPLEX \cite{CPLEX} for quadratic mixed integer problems, as our \BnB solver for our numerical comparison in Section \ref{sec:5_MINLPpaper}.
We refer the reader to \cite{BnB_Masterpaper} for an elaborate overview of the \BnB framework and simply note that \cplexmiqp incorporates many algorithmic features lately developed to improve \BnB performance.


\subsection{Numerical setting and parameter choices}
\label{sec:4_4_MINLPpaper}

We present the setting in which the numerical experiments will be conducted including default parameter choices for the algorithms. If different choices are utilized, it will be mentioned.

We choose $\Omega := [0,1]^2$ as our computational domain. Regarding the Gaussian sources defined in \eqref{eq:gaussian_sources_MINLPpaper}, we choose $l=100$ sources with centers $\tilde{x}_1,\dots , \tilde{x}_l$ being arranged in a uniform $10\times 10$ grid \REV{with step size $\frac{1}{11}$ (resulting in a radius of $\frac{1}{10}$ for Definition \ref{def:adjind_MINLPpaper}).}
The height of the sources is $\kappa = 100$ and the width $\omega$ is chosen such that every source takes $5\%$ of its center-value at a neighboring center. We mention that this choice of height and width is motivated by \cite[Section 4.2]{wesselhoeft2017mixed}. The PDE \eqref{eq:PDE_MINLPpaper} is discretized using uniform piecewise linear finite elements with a step size of $2^{-7}$ (unless specified otherwise) resulting in $N=16641$ vertices.

Whenever a local solver is required, i.e, in Algorithms \ref{algo:Perturbation} and \ref{algo:some_Algorithm}, we use the IPM derived in Section \ref{sec:4_2_MINLPpaper}. 
The outer interior point iteration is stopped as soon as either
$\max\{\norm{\xi_p}_2, \norm{\xi_d}_2, \norm{\xi_c}_2\} \leq 10^{-6}$
or the safeguard $\mu \leq 10^{-15}$ is triggered.
Furthermore, starting from an initial $\mu=1$ we decrease $\mu$ by the factor $0.1$ in each outer interior point iteration. 
The inexactness is implemented by stopping GMRES when the norm of the unpreconditioned relative residual is below $\eta = \max\{\min\{10^{-1}, \mu\}, 10^{-10}\}$. Finally, the diagonal block $ -\frac 2 \varepsilon I_l + \Theta_u $ in the Newton system (\ref{eq:NewtonSystem}) is kept positive definite by setting any negative values to $\gamma = 10^{-6}$ and the preconditioner proposed at the end of Section \ref{sec:4_2_MINLPpaper} is applied by performing the Cholesky decomposition of both $M$ and $K$ once at the beginning of the IPA process.

As initial guess for Algorithms \ref{algo:some_Bologna_Algorithm} and \ref{algo:some_Algorithm} the solution of \eqref{eq:MINLP_contrelax_MINLPaper} obtained by our IPM is used. Do note that this is not necessary since \eqref{eq:MINLP_penalty} for large enough $\varepsilon^0$ is usually still a convex problem in the first iteration of these algorithms so that any initial guess would be sufficient. 
Further default parameters are $\varepsilon^0 = 10^5$ for both algorithms as well as $\sigma = 0.9$ for Algorithm \ref{algo:some_Algorithm} and $\sigma = 0.7$ for Algorithm \ref{algo:some_Bologna_Algorithm}. The more conservative value of $\sigma$ for Algorithm \ref{algo:some_Algorithm} is necessary here, since with $\sigma$ being closer to $0$ one would risk increasing the amount of penalization too fast and thus possibly 'skipping' a good local minimum and settling for an unsatisfactory local minimum. Finally, we used the feasibility tolerance $\efeas = 0.1$.

Regarding \cplexmiqp, we use default options except that we set a time limit of $1$ hour (unless specified otherwise) and a memory limit of $16000$ megabytes for the search tree.
All experiments were conducted on a PC with 32 GB RAM and a QUAD-Core-Processor INTEL-Core-I7-4770 (4x 3400MHz, 8 MB Cache) utilizing Matlab 2019a via which CPLEX 12.9.0 was accessed.

\section{Numerical Experiments}
\label{sec:5_MINLPpaper}

We begin with two different experiments for our Poisson model problem \eqref{eq:MINLP_MINLPpaper} and then shortly discuss a convection-diffusion problem as well as the behaviour of our local solver. 

\subsection{Poisson model problem}
\label{sec:5_1_MINLPpaper}
In the first experiment we want to see that the IPA can indeed handle large-scale problems and convince ourselves that \cplexmiqp, the \BnB method of CPLEX introduced in Section \ref{sec:4_3_MINLPpaper}, can not handle large-scale problems.
In the second experiment we then carry out a detailed comparison of the IPA with the solution strategies presented in Section \ref{sec:4_3_MINLPpaper}.
We further mention that two more experiments were conducted that can be found in a previous version of this article\footnote{\url{https://arxiv.org/abs/1907.06462v2}}:
\begin{itemize}
	\item A parameter study for the IPA with respect to $\pmax\in\N$ and $\theta\in\N$ was conducted upon which $\pmax = 300$ and $\theta=3$ have been selected.
	\item The stochastic robustness of the IPA was investigated, i.e., how robust the solution quality and time is with respect to the random choices made. It was found that while different minima may be found by the IPA for the same problem instance, the minima are all of high quality and the difference in solution time is neglectable.
\end{itemize}

Do note that due to the different implementation languages included in these experiments, the reported computational times only give a qualitative information on the performance of the solvers.

In general, we create an instance of our optimal control problem by generating a desired state $\ybar$ that is a solution of (the discretized version of) \eqref{eq:PDE_MINLPpaper} with $S$ active sources in the right-hand side and the centers of these sources are randomly distributed over $\Ot = [0.1,0.9]^2$. The height and width of these sources coincide with the values that were used for the source-grid in Section \ref{sec:4_4_MINLPpaper}. Clearly, the combinatorial complexity of the optimization problem corresponding to such a desired state increases drastically for larger values of $S$.
To further illustrate the optimization problem here, Figure \ref{fig:figure1_MINLPpaper} exemplarily shows two desired states, one for $S=3$ and one for $S=20$, where the white stars depict $\tilde{x}_1,\dots , \tilde{x}_l$, the centers of the source grid introduced in Section \ref{sec:4_4_MINLPpaper}, and the red stars depict the centers of the $S$ active sources in $\ybar$.

\begin{figure}[ht]
	\centering
	\includegraphics[width=0.3\textwidth]{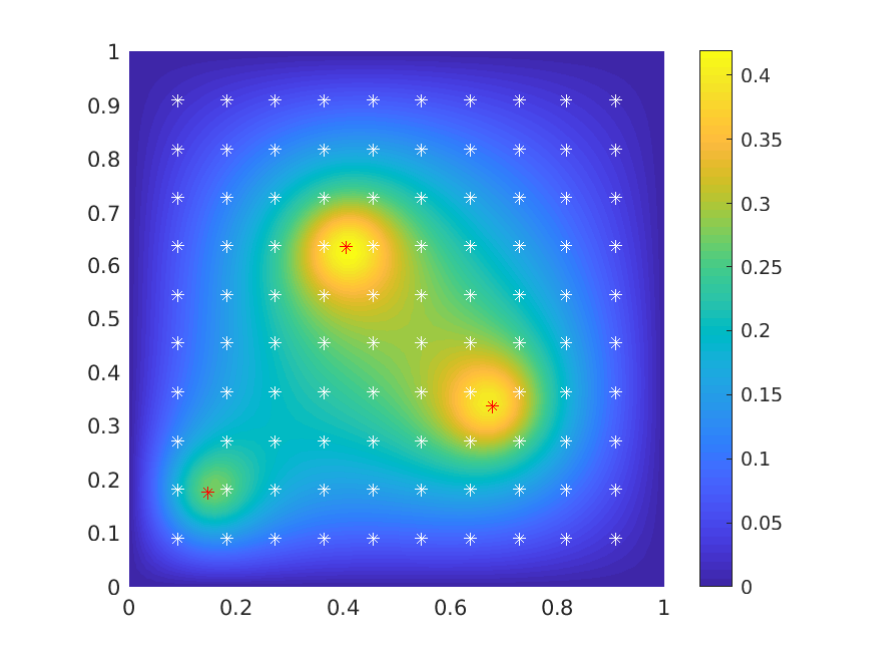}\qquad
	\includegraphics[width=0.3\textwidth]{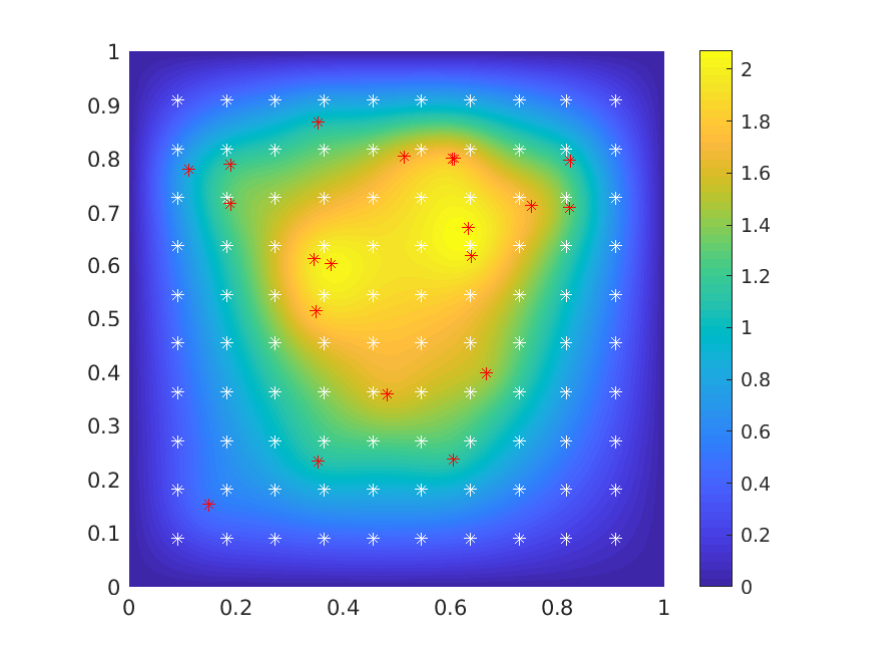}
	\caption{Exemplary desired states with $S=3$ active sources (left) and $S=20$ active sources (right) including the centers of the source grid (white stars) and the centers of the active sources of the respective desired state (red stars).}
	\label{fig:figure1_MINLPpaper}
\end{figure}

\paragraph{First experiment} 
We want to see that the IPA can handle large-scale problems and \cplexmiqp, the \BnB routine of CPLEX, can not. Therefore, we create a problem instance per value of $S\in\{3,10,20\}$ and per step-size $h\in\{2^{-7},2^{-8}\}$ of the FEM grid and solve each instance with the IPA, \cplexmiqp with a $1$ hour time limit, \cplexmiqp with a $10$ hour time limit, and (for comparison reasons) with the simple penalty approach from Algorithm \ref{algo:some_Algorithm}. Regarding the solution quality, the algorithm with the lowest objective function value is indicated with a '$\min$' in Table \ref{table:second_experiment_MINLPpaper} (or a '$\min^*$' if it was the global minimum) and for each other algorithm the relative error towards this minimum objective function value is then displayed. Furthermore, Table \ref{table:second_experiment_MINLPpaper} contains the run times in seconds for each algorithm in each instance, where in case of \cplexmiqp~ 'TL' indicates that the respective time limit was reached.

\begin{table}[ht]
	\centering
	\resizebox{\textwidth}{!}{%
		\begin{tabular}{c|c c|c c|c c|c c|c c|c c}
			\toprule
			\multicolumn{1}{c}{h} & \multicolumn{6}{|c|}{$2^{-7}$} & \multicolumn{6}{|c}{$2^{-8}$}\\
			\midrule
			\multicolumn{1}{c}{S} & \multicolumn{2}{|c|}{3} & \multicolumn{2}{|c|}{10} & \multicolumn{2}{|c|}{20} & \multicolumn{2}{|c|}{3} & \multicolumn{2}{|c|}{10} & \multicolumn{2}{|c}{20}\\
			\hline
			& rel\_err & time & rel\_err & time & rel\_err & time & rel\_err & time & rel\_err & time & rel\_err & time\\
			\hline
			\multirow{1}{*}{Penalty} & $\mathrm{\bf{min}}^*$ & 89 & $20\%$ & 163 & $57\%$ & 188 & $\mathrm{\bf{min}}$ & 541 & $8\%$ & 1101 & $12\%$ & 1400 \\
			\multirow{1}{*}{IPA} & $\mathrm{\bf{min}}^*$ & 925 & $13\%$ & 1035 & $\mathrm{\bf{min}}$ & 1143 & $\mathrm{\bf{min}}$ & 6219 & $\mathrm{\bf{min}}$ & 7550 & $\mathrm{\bf{min}}$ & 9190 \\
			\multirow{1}{*}{\cplexmiqp 1h} & $\mathrm{\bf{min}}^*$ & 1527 & $13\%$ & TL & $19\%$ & TL & $6805\%$ & TL & $45535\%$ & TL & $93718\%$ & TL \\
			\multirow{1}{*}{\cplexmiqp 10h} & - & - & $\mathrm{\bf{min}}$ & TL & $1\%$ & TL & $6805\%$ & TL & $45535\%$ & TL & $93718\%$ & TL \\
			\bottomrule
		\end{tabular}
	}
	\caption{Results of the first experiment. For each problem instance the algorithm with the lowest objective function value is indicated. The respective relative error of other algorithms as well as the solution times (in seconds) are furthermore reported.}
	\label{table:second_experiment_MINLPpaper} 
\end{table}

\noindent We observe that for $h=2^{-7}$ and $S=3$ all algorithms find the global minimum, although we stress that this is only a single problem instance which does not allow for a conclusive comparison with respect to solution quality. A more detailed comparison will be carried out in the next experiment. With an increase in $S$ (and thus an increase in the combinatorial complexity of the problem), \cplexmiqp, while hitting the prescribed time limit, is still able to provide good solutions, although our IPA is able to at least keep up.
Refining the FEM-mesh once and thus moving towards $h=2^{-8}$ (resulting in $N=66049$ instead of $N=16641$) results in \cplexmiqp not being able to handle the problem at all. The time limit is always reached and the algorithm (even given 10 hours time)  terminates with a tremendous relative error with respect to the qualitative solutions found by our IPA. The solution found by the IPA is then by construction always better than the solution found by the simple penalty algorithm.
One might be tempted to believe that the simple penalty algorithm could also be a viable alternative due to its inherent fast solution time but the next experiment will reveal that the algorithm cannot produce qualitative points in a reliable way.

\paragraph{Second experiment} 

We carry out a detailed comparison between the IPA, the penalty algorithm in Algorithm \ref{algo:some_Algorithm} and \cplexmiqp. In order to do so, we construct a test set of $20$ problem instances per value of $S\in\{3,6,10,15,20\}$.
We then solve this test set with the algorithms under analysis and compare the results with respect to solution time and quality. For the solution time, we report 't\_av' the average solution time in seconds and for the solution quality, we choose the following two criteria.
\begin{itemize}
	\item 'min\_count': for each desired sate, we check which algorithm achieved the smallest objective function value. This algorithm is then awarded a score. Surely, multiple algorithms can be awarded a score in the same run (when multiple algorithms find the same 'best' minimum).
	\item 'rel\_err\_av': for each desired state, we store for each algorithm the relative error between the objective function value achieved by that algorithm and the smallest objective function value in that run (the one that was awarded a 'min\_count'-score). Only runs resulting in a non-zero relative error are taken into account when computing this average relative error.
\end{itemize}
We chose to measure the quality of the algorithms via the described two quantities since, as the centers of the desired states in the test set are randomly distributed over $\Ot$, the global minimum of the optimization problem is not known analytically. Therefore, the 'min\_count'-value simply tells us how often an algorithm performed best compared to the other algorithms. The average relative error is an additional measure of quality.
The results of this experiment can be found in Table \ref{table:comparison_linGAUS_MINLPpaper}.

\begin{table}[ht]
	\centering
	\resizebox{\textwidth}{!}{%
		\begin{tabular}{c |cc c c c |c c ccc |ccccc }
			\toprule
			\multicolumn{1}{c}{} & \multicolumn{5}{|c|}{t\_av (s)} & \multicolumn{5}{|c|}{min\_count} & \multicolumn{5}{|c}{rel\_err\_av ($\%$)} \\
			\midrule
			S & 3 & 6 & 10 & 15 & 20 & 3 & 6 & 10 & 15 & 20 & 3 & 6 & 10 & 15 & 20\\
			\midrule
			Penalty & 88 & 125 & 152 & 184 & 202 & 12 & 5 & 2 & 1 & 1 & 33 & 41 & 52 & 33 & 56\\
			IPA &  918 & 1112 & 1149 & 1343 & 1239 & 20 & 13 & 14 & 16 & 15 & 0 & 6 & 37 & 11 & 12\\
			\cplexmiqp & 1885 & 3486 & TL & TL & TL & 20 & 18 & 9 & 5 & 5 & 0 & 13 & 5035 & 4311 & 7582\\
			\bottomrule
		\end{tabular}
	}
	\caption{Results of the second experiment. Comparison of the penalty algorithm, the IPA and \cplexmiqp for different values of $S$.}
	\label{table:comparison_linGAUS_MINLPpaper} 
\end{table}

\noindent Starting the discussion with the average time, we observe that \cplexmiqp is heavily affected by an increase in $S$ while the IPA is only slightly affected and the simple penalty algorithm is by construction the fastest.
Moving to the solution quality, we see that \cplexmiqp as well as the IPA always find the global minimum for $S=3$ where the simple penalty algorithm only finds the global minimum in $12$ cases with an average relative error of $33\%$ in the remaining $8$ cases. Increasing $S$ (and thus the combinatorial complexity of the problem), we observe that \cplexmiqp (especially for $S \geq 10$) fails to find the global minimum in the given time. The IPA on the other hand then starts to be the most competitive algorithm in the 'min\_count'-sense, i.e., producing the smallest objective function values compared to the other algorithms (do also note the small relative average error of the IPA).
Furthermore, we report that for $S=10,15$, and $20$ there was always one problem instance where \cplexmiqp only returned the zero solution (which is feasible but does not make sense from the application point of view). As a result the average relative error is significantly large.

\REV{To put the results of this experiment into a better perspective, Figure \ref{fig:Boxplots_fvals_poi_MINLPpaper} contains, for each part of the test set, a boxplot related to the objective function of the final solutions attained by each algorithm (i.e. for each value of $S$ the test set contains $20$ instances, such that for each algorithm a boxplot is created for the $20$ objective function values related to the solutions we found). It is important to note that in the top of Figure \ref{fig:Boxplots_fvals_poi_MINLPpaper} the outliers of the data sets have been removed for visual clarity whilst in the bottom part of Figure \ref{fig:Boxplots_fvals_poi_MINLPpaper} the outliers are contained in the data sets. As a result, the previously described phenomenon becomes clearly visible in the bottom of Figure \ref{fig:Boxplots_fvals_poi_MINLPpaper}: for $S=10$, $15$, and $20$ the data for \cplexmiqp includes an outlier with a significantly larger value such that the remaining box plots are tightly squeezed. 
Even if those outliers play a  role in getting the large average relative errors from Table \ref{table:comparison_linGAUS_MINLPpaper}, when taking a look at the  boxplots related to the data without outliers, we can easily see that the IPA  generally has both a smaller median and a smaller variance for larger $S$.
}

\begin{figure}[ht]
	\centering
	\includegraphics[width=\textwidth, height = 0.5\textheight]{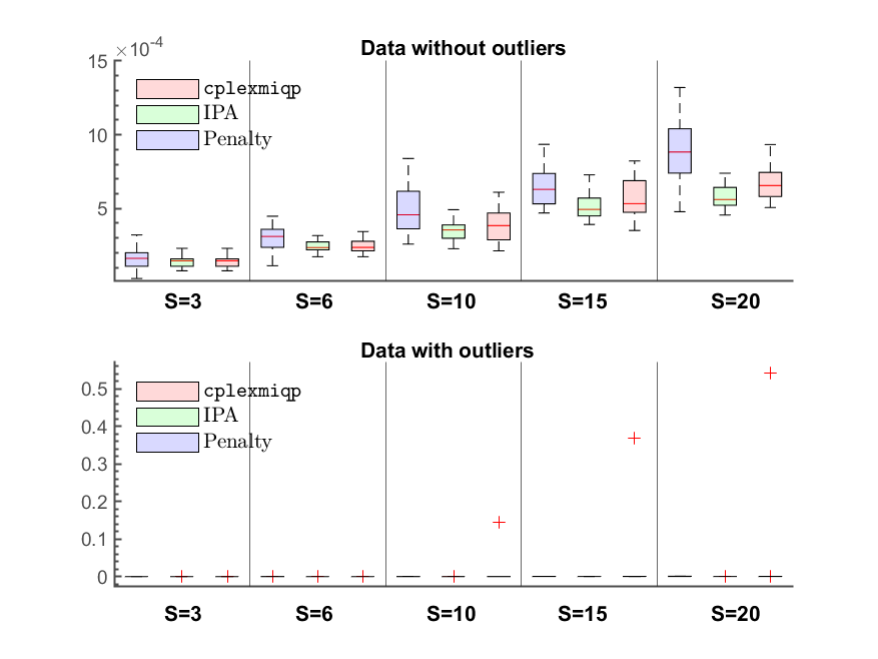}
	\caption{Results of the second experiment: for each part of the test set, a boxplot related to the objective function of the final solutions obtained by each algorithm is depicted. The outliers of the data sets are not included in the top part of the figure.}
	\label{fig:Boxplots_fvals_poi_MINLPpaper}
\end{figure}

\noindent \REV{The results from Figure \ref{fig:Boxplots_fvals_poi_MINLPpaper} further strengthen the observation made in Table~\ref{table:comparison_linGAUS_MINLPpaper}: the objective function values obtained with the points found by the IPA are, on the one hand, clearly better than the ones found by the Penalty algorithm and, on the other hand, either very similar (for $S=3$ and $S=6$) 
or superior  (for larger $S$)
to the ones found by \cplexmiqp.}


 Combining the results of this experiment with the results from the first experiment, we can conclude that the IPA can solve large-scale problems, and can, at the same time, compete with \cplexmiqp in smaller problem instances. The simple penalty approach is very fast but, as we can see in this experiment, fails to produce solutions of high quality in a reliable fashion.

\subsection{Convection-Diffusion model problem}
\label{sec:5_2_MINLPpaper}
We now consider the original optimal control problem, but governed by the convection-diffusion PDE
\begin{align}
\label{eq:PDE_CD_MINLPpaper}
-\Delta y(x) + w(x)\cdot\nabla y(x) = \sum_{i=1}^l u_i \chi_i(x),\quad x\in\Omega,
\end{align}
with the wind vector $w(x) = (2x_2(1-x_1^2), -2x_1(1-x_2^2))\transp$ and piecewise constant source functions $\chi_1,\dots ,\chi_l\in L^2(\Omega)$, that are constant on the subdomains $\Omega_1,\dots \Omega_l\subset\Omega$ forming a uniform decompostion of $\Omega = [0,1]^2$ into $l$ many squares. Here, we use Q1 finite elements, while also employing the Steamline Upwind Petrov-Galerkin (SUPG) \cite{brooks1982streamline} upwinding scheme as implemented in the {\tt IFISS} software package \cite{elman2007algorithm} to discretize \eqref{eq:PDE_CD_MINLPpaper} and build the relevant finite element matrices.

For the resulting discretized optimal control problem, we repeat the second experiment from the previous section, where all settings and parameters are chosen as before. We chose not to include the other experiments to keep the length of this presentation healthy but can report that results similar to the Poisson problem are obtained. The result of the second experiment for this convection-diffusion problem can be seen in Table \ref{table:comparison_linCDPWconst_MINLPpaper}.

\begin{table}[ht]
	\centering
	\resizebox{\textwidth}{!}{%
		\begin{tabular}{c |cc c c c |c c ccc |ccccc }
			\toprule
			\multicolumn{1}{c}{}& \multicolumn{5}{|c|}{t\_av (s)} & \multicolumn{5}{|c|}{min\_count} & \multicolumn{5}{|c}{rel\_err\_av ($\%$)} \\
			\midrule
			S & 3 & 6 & 10 & 15 & 20 & 3 & 6 & 10 & 15 & 20 & 3 & 6 & 10 & 15 & 20\\
			\midrule
			Penalty & 103 & 148 & 198 & 231 & 247 & 16 & 7 & 4 & 0 & 0 & 39 & 27 & 44 & 50 & 38\\
			IPA &  943 & 1008 & 1083 & 1223 & 1337 & 19 & 16 & 15 & 15 & 14 & 16 & 19 & 12 & 16 & 12\\
			cplexmiqp & 1937 & TL & TL & TL & TL & 20 & 16 & 7 & 5 & 6 & 0 & 10 & 24 & 52 & 24\\
			\bottomrule
		\end{tabular}
	}
	\caption{Results for the convection-diffusion problem: comparison of the penalty algorithm, the IPA and \cplexmiqp for different values of $S$.}
	\label{table:comparison_linCDPWconst_MINLPpaper} 
\end{table}


\noindent Investigating Table \ref{table:comparison_linCDPWconst_MINLPpaper}, we observe that \cplexmiqp shows basically the same behaviour as in the Poisson problem: it is always able to solve the problem in the given time for $S=3$, but then requires much more time and starts to produce unsatisfactory solutions for larger values of $S$. The IPA again succeeds in finding either the global minimum or a reasonable solution in around $15-20$ minutes. The simple penalty approach is again very fast, but also quite unreliable in terms of solution quality.

\REV{As in the previous section, Figure \ref{fig:Boxplots_fvals_CD_MINLPpaper} contains, for each part of the test set, a boxplot for the objective function related to the final solutions found by each algorithm. While \cplexmiqp does not produce any clear outliers for this data set, the results still verify that the IPA is the superior algorithm in this comparison.}

\begin{figure}[ht]
	\centering
	\includegraphics[width=\textwidth, height = 0.25\textheight]{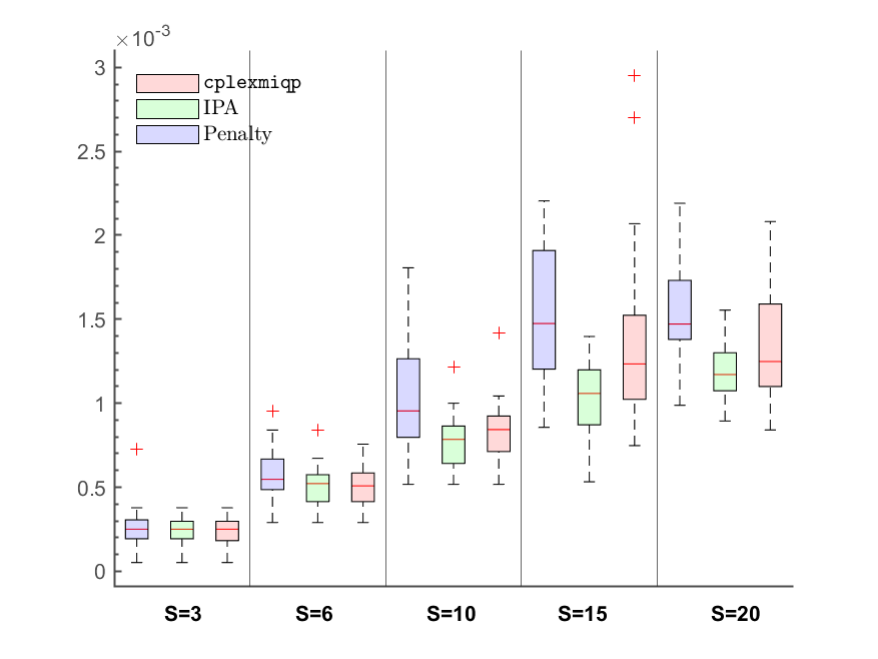}
	\caption{Results of the second experiment: for each part of the test set, a boxplot related to the objective function of the final solutions attained by each algorithm is depicted. For $S=10$, $15$, and $20$ the boxplot for \cplexmiqp is plotted with respect to the right y-axis.}
	\label{fig:Boxplots_fvals_CD_MINLPpaper}
\end{figure}

\noindent \REV{We conclude this numerical comparison with a final experiment that shall, on the one hand, highlight the robustness of the optimization results with respect to the FEM mesh and, on the other hand, show the efficiency of our numerical linear algebra. We create one problem instance with $S=10$ active sources and discretize this instance with a decreasing FEM mesh size of $h=2^{-4},~2^{-5},~2^{-6},~2^{-7}$. The instance is then solved for each mesh size with the different algorithms from the second experiment and additionally with a version of the IPA that does not embed the preconditioned \gmres described in Section \ref{sec:4_2_MINLPpaper}. This is done once for the Poisson model problem and once for the convection-diffusion model problem. Figure \ref{fig:step11_MINLPpaper} shows the final objective function value obtained (top row) and the required computational times (bottom row) for the Poisson problem (left row) and the convection-diffusion problem (right column).}


\begin{figure}[ht]
	\centering
	\includegraphics[width=0.9\textwidth]{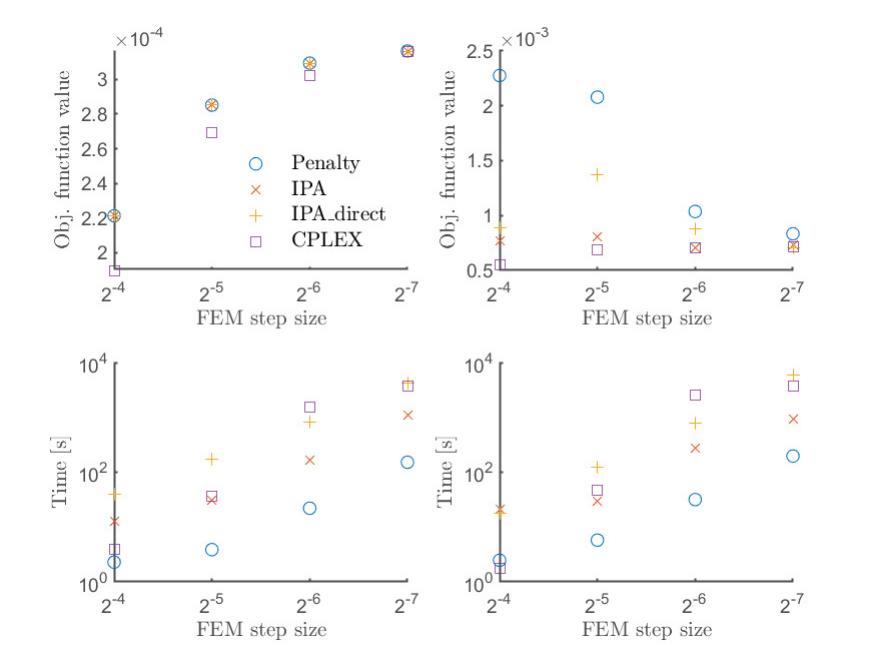}
	\caption{Results of the final experiment: the final objective function values obtained (top row) and the required computational times (bottom row) for the Poisson problem (left row) and the convection-diffusion problem (right column).}
	\label{fig:step11_MINLPpaper}
\end{figure}

\noindent \REV{It can be observed that for both model problems, the final objective function values obtained with the various algorithms increasingly agree with each refinement of the FEM mesh. Furthermore, the bottom row of Figure \ref{fig:step11_MINLPpaper} shows the efficiency of our numerical linear algebra since the version of the IPA using only a direct solver requires significantly more time, especially for finer FEM mesh sizes.}

\subsection{Analysis of the local solver}
\label{sec:5_3_MINLPpaper}
As already mentioned, one of the main benefits of our IPA is the possibility to exploit the problem features through \REV{an efficient implementation of the} 
local solver in line $3$ of Algorithm \ref{algo:Perturbation}. We now briefly report on the numerical behaviour of our implementation of the IPM described in Section \ref{sec:4_2_MINLPpaper}.
Thus, we create an exemplary problem instance (both for the Poisson and convection-diffusion problem) for $S=10$, and vary the step size of the FEM grid as $h\in\{2^{-5},2^{-6}\}$. The instance is then solved for each step size with the IPA, where the settings for the IPM and the IPA are as before.
Figure \ref{fig:gmres} shows the number of nonlinear (outer) iterations (NLI) required by the IPM and the average number of preconditioned GMRES iterations (aGMRES) for each value of $\varepsilon$ visited during the IPA. 
Clearly, multiple values reported for a single value of $\varepsilon$ correspond to active perturbation cycles of Algorithm \ref{algo:Perturbation}. 

\begin{figure}[ht] 
	\centering		
	\begin{tikzpicture}
	\begin{semilogxaxis}[only marks,width=0.5\textwidth, 
	legend pos=south east,
	legend cell align=left,
	legend style={font=\small},
	height = .2 \textheight,
	xmax = 100000
	]
	\foreach \j in {1,2} {
		\addplot+ table[x index = 0, y index = \j] {./poi_medium.dat};
	}
	\legend{aGMRES,NLI};
	\end{semilogxaxis}
	\end{tikzpicture}		
	\begin{tikzpicture}
	\begin{semilogxaxis}[only marks,width=0.5\textwidth, 
	legend pos=south east,
	legend cell align=left,
	legend style={font=\small},
	height = .2 \textheight,
	xmax = 100000
	]
	\foreach \j in {1,2} {
		\addplot+ table[x index = 0, y index = \j] {./cd_medium.dat};
	}
	\legend{aGMRES,NLI};
	\end{semilogxaxis}
	\end{tikzpicture}		
	\begin{tikzpicture}
	\begin{semilogxaxis}[only marks,width=0.5\textwidth, 
	legend pos=south east,
	legend cell align=left,
	legend style={font=\small},
	height = .2 \textheight,
	xmax = 100000
	]
	\foreach \j in {1,2} {
		\addplot+ table[x index = 0, y index = \j] {./poi_fine.dat};
	}
	\legend{aGMRES,NLI};
	\end{semilogxaxis}
	\end{tikzpicture}		
	\begin{tikzpicture}
	\begin{semilogxaxis}[only marks,width=0.5\textwidth, 
	legend pos=south east,
	legend cell align=left,
	legend style={font=\small},
	height = .2 \textheight,
	xmax = 100000
	]
	\foreach \j in {1,2} {
		\addplot+ table[x index = 0, y index = \j] {./cd_fine.dat};
	}
	\legend{aGMRES,NLI};
	\end{semilogxaxis}
	\end{tikzpicture}
	\caption{Number of IPM iterations and average GMRES iterations during the IPA steps for the Poisson (left) and convection-diffusion (right) problems for FEM step-lengths $h=2^{-5}$ (top) and $h=2^{-6}$ (bottom) over the penalty parameter $\varepsilon$.}
	\label{fig:gmres}
\end{figure}

\noindent Firstly, we observe that both values of NLI and aGMRES are higher at the beginning of the IPA process, that is for larger values of $\varepsilon$. On the other hand, when $\varepsilon$ gets smaller and more perturbation cycles are expected, the number of IPM iterations may get lower and, mostly, the average number of GMRES iterations is reduced. This shows that the IPA together with the IPM efficiently drives the solution of problem \eqref{eq:MINLP_penalty} to the mixed-integer solution of the original problem. This behaviour is observed in Figure \ref{fig:gmres} for both problems and the varying mesh sizes.

Secondly, the reported number of average number of GMRES iterations is pretty low and does not depend on the mesh size. This reveals the effectiveness of the proposed preconditioner also in combination with the inexact approach. Remarkably, values of 
aGMRES are extremely low in the last IPA iterations when $\varepsilon$ is small.

\section{Conclusion \& Outlook}
\label{sec:6_MINLPpaper}
A standard MIPDECO problem with a linear PDE constraint and a modelled control was presented and discretized. A novel improved penalty algorithm (IPA) was developed, that combines well-known exact penalty approaches with a basin hopping strategy and an updating tool for the penalty parameter. As a result, only a local optimization solver is required and an interior point method (IPM) that is suited for the problem in question was presented. The linear algebra phase of the IPM was handled by a Krylov space method together with an efficient preconditioner. Via this, the IPA was shown to work very well in numerical applications for a Poisson as well as a convection-diffusion problem when compared to a simple penalty approach and \cplexmiqp, the \BnB routine of CPLEX. 
As an outlook, the authors want to mention that the IPA has already been successfully applied to the presented optimal control problem, but governed by the nonlinear PDE
\begin{align*}
-\Delta y(x) + y(x)^2 = \sum_{i=1}^l u_i \phi_i(x),\quad x\in\Omega,
\end{align*}
where again the Gaussian source functions defined in \eqref{eq:gaussian_sources_MINLPpaper} are used and the IPM has been adapted as described in Remark \ref{rem:nonlinIPM_MINLPpaper}. As first results, Figure \ref{fig:figure1_nonlin_MINLPpaper} shows the desired state of a random problem instance for $S=10$, as well as the optimal state found by the IPA and a difference plot. Furthermore, Figure \ref{fig:figure2_nonlin_MINLPpaper} shows the result of the experiment from the previous Section \ref{sec:5_3_MINLPpaper} conducted for this problem instance.

\begin{figure}[ht]
	\centering
	\includegraphics[width=0.3\textwidth]{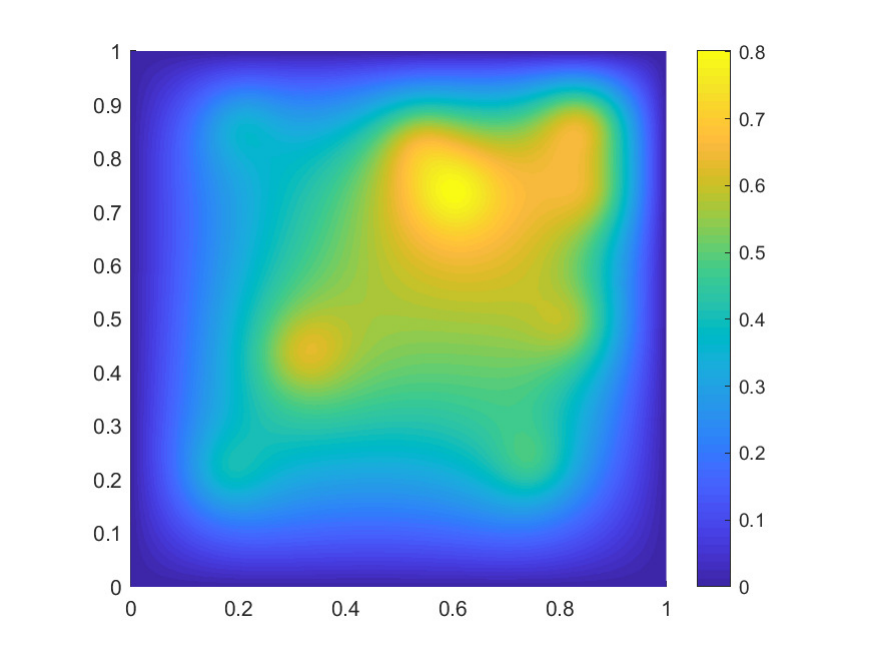}\hfill
	\includegraphics[width=0.3\textwidth]{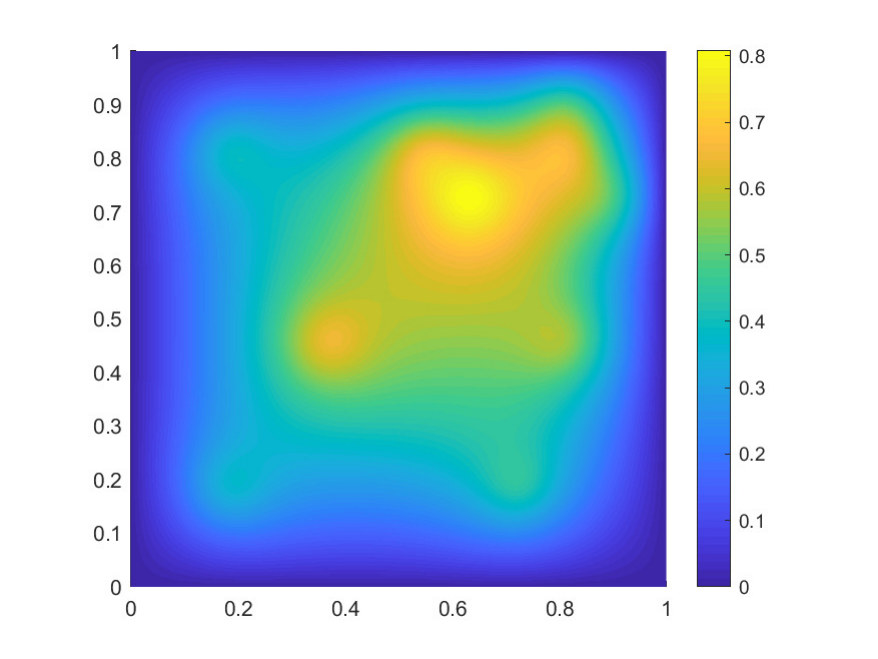}\hfill
	\includegraphics[width=0.3\textwidth]{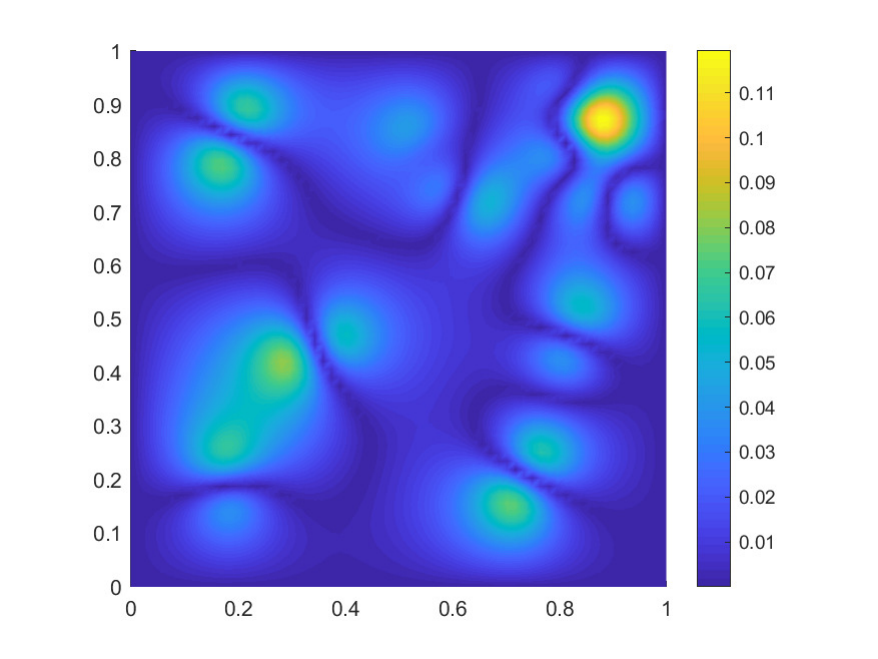}
	\caption{Desired state (left), optimal state found by the IPA (middle), and difference plot (right) for a problem instance for $S=10$ of the nonlinear problem.}
	\label{fig:figure1_nonlin_MINLPpaper}
\end{figure}

\begin{figure} \centering
	\begin{tikzpicture}
	\begin{semilogxaxis}[only marks,width=0.5\textwidth, 
	legend pos=south east,
	legend cell align=left,
	legend style={font=\small},
	height = .2 \textheight,
	xmax = 100000
	]
	\foreach \j in {1,2} {
		\addplot+ table[x index = 0, y index = \j] {./Data_nonlin_IPM_h_2p5.dat};
	}
	\legend{aGMRES,NLI};
	\end{semilogxaxis}
	\end{tikzpicture}	
	\begin{tikzpicture}
	\begin{semilogxaxis}[only marks,width=0.5\textwidth, 
	legend pos=south east,
	legend cell align=left,
	legend style={font=\small},
	height = .2 \textheight,
	xmax = 100000
	]
	\foreach \j in {1,2} {
		\addplot+ table[x index = 0, y index = \j] {./Data_nonlin_IPM_h_2p6.dat};
	}
	\legend{aGMRES,NLI};
	\end{semilogxaxis}
	\end{tikzpicture}
	\caption{Number of IPM iterations and average GMRES iterations during the IPA steps for the nonlinear problem instance for FEM step-lengths $h=2^{-5}$ (left) and $h=2^{-6}$ (right) over the penalty parameter $\varepsilon$.}
	\label{fig:figure2_nonlin_MINLPpaper}
\end{figure}
\noindent Overall, these results are already very encouraging and in future work, a comparison of the IPA with state of the art solvers for such nonlinear problems should be carried out (do note that CPLEX cannot deal with nonlinear PDE constraints). 
Furthermore, future work shall contain the application to MIPDECO problems that are governed by time-dependent PDEs  \REV{(as these result in a truly large-scale context) as well as the extension from binary to general integer constraints and the development of strategies to efficiently deal with these.} 

\section*{Acknowledgement}
D. Garmatter and M. Stoll acknowledge the financial support by the Federal Ministry of Education and Research of Germany (support code 05M18OCB). D. Garmatter, M. Porcelli, and M. Stoll were partially supported by the DAAD-MIUR Joint Mobility Program 2018-2020 (Grant 57396654). The work of M. Porcelli
was also partially supported by the National Group of Computing Science (GNCS-INDAM).


 \printbibliography
\end{document}